\documentclass[12pt]{article}
\usepackage{amssymb,amsmath,amsthm,latexsym}
\usepackage[cp1251]{inputenc}
\usepackage{dsfont}
\usepackage{amssymb}
\usepackage{indentfirst}
\usepackage{graphicx}
\usepackage{tikz-cd}
\usepackage{caption}
\usepackage{multicol}
\usepackage{enumitem}
\usepackage{hyperref}
\hypersetup{citecolor=red,}
\usepackage{amsfonts, array, hhline}
\usepackage{color}
\usepackage{lmodern}
\usepackage[left=1in,right=1in,top=1in,bottom=1.5in]{geometry}

\makeatletter
\theoremstyle{definition}
\newtheorem*{rep@theorem}{\rep@title}
\newcommand{\newreptheorem}[2]{%
\newenvironment{rep#1}[1]{%
 \def\rep@title{#2 \ref{##1}}%
 \begin{rep@theorem}}%
 {\end{rep@theorem}}}
\makeatother

\numberwithin{equation}{section}

\theoremstyle{definition}
\newtheorem{dfn}{Definition}[section]

\newtheorem{thm}[dfn]{Theorem}
\newtheorem{lm}[dfn]{Lemma}

\newtheorem{crl}[dfn]{Corollary}

\newreptheorem{lm}{Lemma}

\theoremstyle{remark}
\newtheorem{rmk}[dfn]{Remark}

\newcommand{\e}{\varepsilon}
\newcommand{\pt}{\partial}
\newcommand{\mc}[1]{\mathcal{#1}}
\newcommand{\mf}[1]{\mathfrak{#1}}
\newcommand{\ol}[1]{\overline{#1}}
\newcommand{\ra}{\rightarrow}
\newcommand{\R}{\mathbb{R}}
\renewcommand{\H}{\mathbb{H}}
\renewcommand{\S}{\mathbb{S}}
\newcommand{\C}{\mathbb{C}}
\newcommand{\E}{\mathbb{E}}

\newcommand{\tr}{{\rm tr}}

\newcommand{\conv}{{\rm conv}}

\newcommand{\dev}{{\rm dev}}
\newcommand{\inter}{{\rm int}}

\newcommand{\im}{{\rm im}}
\newcommand{\ost}{{\rm ost}}
\newcommand{\ad}{{\rm ad}}
\newcommand{\Ad}{{\rm Ad}}
\newcommand{\dS}{d\mathbb{S}}
\newcommand{\MCG}{\textrm{MCG}}

\renewcommand{\hat}{\widehat}
\renewcommand{\tilde}{\widetilde}
\renewcommand{\Re}{{\rm Re}}
\renewcommand{\Im}{{\rm Im}}

\title{Dual metrics on the boundary of strictly polyhedral hyperbolic 3-manifolds}
\author{Roman Prosanov \thanks{This research was funded in part by the Swiss National Science Foundation grant $200021_-179133$ and in part by the Austrian Science Fund (FWF) \url{https://doi.org/10.55776/ESP12}. For open access purposes, the author has applied a CC BY public copyright license to any author-accepted manuscript version arising from this submission.}}
\date{}

\begin{document}
\maketitle
\abstract{Let $M$ be a compact oriented 3-manifold with non-empty boundary consisting of surfaces of genii $>1$ such that the interior of $M$ is hyperbolizable. We show that for each spherical cone-metric $d$ on $\partial M$ such that all cone-angles are greater than $2\pi$ and the lengths of all closed geodesics that are contractible in $M$ are greater than $2\pi$ there exists a unique strictly polyhedral hyperbolic metric on $M$ such that $d$ is the induced dual metric on $\partial M$.}

\section{Introduction}

\subsection{Metric geometry of convex polyhedra}

The class of problems dealing with determining a surface in space by some associated geometric data has a rich history. This paper constitutes another its chapter, studying a discrete analogue of the third fundamental form, \emph{the dual metric}, on the boundary of polyhedral hyperbolic 3-manifolds. We start from a basic definition.

\begin{dfn}
\label{conedef}
Let $S$ be a closed surface. We say that $d$ is a \emph{Euclidean cone-metric} (resp. \emph{spherical} or \emph{hyperbolic}) on $S$ if it is locally isometric to the Euclidean plane (resp.  the standard sphere or the hyperbolic plane) except at finitely many points where it is locally isometric to a Euclidean cone (resp. spherical or hyperbolic) of total angle $\neq 2\pi$. We say that $d$ is \emph{convex} if all cone angles are $<2\pi$. We say that $d$ is \emph{concave} if all cone angles are $>2\pi$. 
\end{dfn}

A theorem of Alexandrov~\cite[Chapter VI]{Ale} states that a convex Euclidean cone-metric on the 2-sphere is realized as the induced intrinsic metric on the boundary of a unique up to isometry compact convex polyhedron in the Euclidean 3-space $\E^3$. The induced intrinsic metric here means the induced path metric and the realization might degenerate to the double cover of a convex polygon. This is a polyhedral version of a fundamental result in differential geometry: every Riemannian metric of positive Gaussian curvature on the 2-sphere is realized as the induced metric on the boundary of a unique up to isometry smooth convex body in $\E^3$. This completely characterizes the induced metrics on the boundaries of smooth strictly convex bodies, i.e., those with a positive-definite shape operator at every boundary point. The existence part is known as the famous \emph{Weyl Problem} resolved by Nirenberg~\cite{Nir} from one side, and by Alexandrov and Pogorelov~\cite{Ale, Pog} from another side. The uniqueness part is known as the global rigidity of closed smooth strictly convex surfaces in $\E^3$. We refer to Cohn-Vossen~\cite{CoV} and to Herglotz~\cite{Her} as to one of the main contributors. 

Alexandrov noticed that his theorem can be transferred without changes to polyhedra in the standard 3-sphere $\S^3$ and in the hyperbolic 3-space $\H^3$. Convex bodies in these spaces have another interesting associated metric structure, which we call \emph{the dual metric}, and which is somewhat degenerate in the Euclidean 3-space $\E^3$. In the standard 3-sphere $\S^3$ there is a natural duality between points and oriented 2-planes. For a convex body $C \subset \S^3$ this allows to define \emph{the Gauss map} as a multivalued map on $\partial C$ sending $p \in \partial C$ to the set of points dual to supporting planes to $C$ at $p$ oriented outwards $C$. The image of $\partial C$ under this map is a surface bounding another closed convex set $C^* \subset \S^3$, which we call \emph{the dual} of $C$. We can also define $C^*$ as the set of points dual to all 2-planes not intersecting $\inter(C)$ oriented outwards $C$.

This duality is polar on closed convex sets, i.e., $(C^*)^*=C$. It follows from the Alexandrov Theorem for $\S^3$ that a convex polyhedron $C \subset \S^3$ is also uniquely determined by the induced intrinsic metric on the boundary of $C^*$. This does not work the same in Euclidean 3-space: the natural metric structure on \emph{the co-Euclidean space}, which is the space of oriented 2-planes in $\E^3$, is degenerate, and the induced intrinsic metric on the dual convex set is always isometric to the standard 2-sphere $\S^2$. Parallel translations allow to consider $\S^2$ as the target for the Gauss map on convex surfaces in $\E^3$. We refer to~\cite{FS} for a more detailed exposition. 

The situation is more interesting in the hyperbolic 3-space $\H^3$. The set of oriented 2-planes in $\H^3$ is naturally identified with \emph{the de Sitter 3-space} $\dS^3$, which is a model Lorentzian 3-space of constant curvature 1. See Section~\ref{minkow} for more details on this duality. A non-trivial work was required to understand metric properties of dual convex bodies in this setting. For the case of compact convex polyhedra in $\H^3$ this resulted in the following theorem of Rivin--Hodgson~\cite{HR}:

\begin{dfn}
A metric on a surface $S$ is called \emph{large} if the lengths of all closed contractible geodesics are greater than $2\pi$.
\end{dfn}

\begin{thm}
\label{HR}
For every concave large spherical cone-metric $d$ on the 2-sphere $S$ there is a unique up to isometry compact convex polyhedron $C \subset \H^3$ such that $(S, d)$ is isometric to the boundary of its dual $C^* \subset \dS^3$. 
\end{thm}

One can see that the conditions on $d$ in Theorem~\ref{HR} are also necessary (the only non-trivial condition is largeness, and its necessity is proven in~\cite{HR}). Theorem~\ref{HR} can be seen as a generalization of the Andreev Theorem~\cite{And} on compact convex polyhedra in $\H^3$ with non-obtuse dihedral angles, another generalization of which was famously used by Thurston in his Hyperbolization Theorem for Haken 3-manifolds, see~\cite{Thu, Mor}. For a corrected proof of the Andreev Theorem and more references we refer to~\cite{RHD}. Theorem~\ref{HR} was also used by Rivin in~\cite{Riv2} to describe the dihedral angles of ideal polyhedra in $\H^3$, which in turn implied a resolution of the Steiner Problem. The latter asked how to distinguish the combinatorial types of convex Euclidean 3-polyhedra that admit a realization with all vertices on the standard 2-sphere. A more general Steiner Problem, with the 2-sphere replaced by the other types of quadrics in $\mathbb{RP}^3$, was resolved in~\cite{DMS}. In~\cite{Sch} Schlenker proved a version of Theorem~\ref{HR} for smooth strictly convex bodies in $\H^3$. We note that if $C \subset \H^3$ is a smooth strictly convex body, then the pull-back of the dual metric by the Gauss map is exactly \emph{the third fundamental form} of $\partial C$. 

\subsection{Hyperbolic 3-manifolds with convex boundary}

In the seventies Thurston revolutionized the field of 3-dimensional topology by showing the importance of homogeneous geometric structures on 3-manifolds, posing his famous Geometrization Conjecture (whose proof was completed by Perelman) and demonstrating a special role of hyperbolic geometry in dimension 3. In this paper we will be interested in studying compact hyperbolic 3-manifolds with convex polyhedral boundary. 

Let $M$ be an oriented smooth compact 3-manifold with non-empty boundary and $N$ be a 3-manifold homeomorphic to the interior of $M$.

\begin{dfn}
A subset $C$ of a complete hyperbolic manifold is called \emph{totally convex} if it is closed and contains every geodesic segment between any two points of $C$.
\end{dfn}

\begin{dfn}
A hyperbolic metric $\overline g$ on $N$ is called \emph{convex cocompact} if it is complete and $(N, \overline g)$ contains a compact totally convex subset. 
\end{dfn}

We assume that $N$ admits a convex cocompact hyperbolic metric. Then $M$ is irreducible and atoroidal. We also assume that $M$ is distinct from the 3-ball and the solid torus. Then all its boundary components have genii at least two. Thurston's Hyperbolization Theorem implies that the converse holds: the interior of a compact 3-manifold satifying these conditions admits a convex cocompact hyperbolic metric, see~\cite{Mar2, Mor}. We will call such $M$ \emph{admissible}. Convex cocompact hyperbolic metrics are generic as they constitute the interior of the moduli space of complete hyperbolic metrics on $N$. 

Let $M$ be endowed with a hyperbolic metric $g$ such that the boundary is locally convex. 
A developing map ${\rm dev}: \tilde M \rightarrow \H^3$ embeds isometrically the universal cover $(\tilde M, g)$ to $\H^3$ as a (non-compact) convex set equivariantly with respect to the holonomy map $\rho: \pi_1(M) \rightarrow G$, where $G$ is the group of orientation-preserving isometries of $\H^3$. We look at the dual convex set $({\rm dev}(\tilde M))^* \subset \dS^3$ defined in Section~\ref{minkow}. The group $G$ can be also identified with the identity component of the isometry group of $\dS^3$, and $({\rm dev}(\tilde M))^*$ is, of course, also invariant with respect to $\rho$. However, the action may not be free on $({\rm dev}(\tilde M))^*$. But under mild restrictions (positive distance from $\partial M$ to the convex core of $(M, g)$, which will be defined later) it is free and properly discontinuous on $\partial({\rm dev}(\tilde M))^*$. By taking the quotient, we obtain \emph{the dual metric} on $\partial M$. Again, in the case of smooth strictly convex boundary, it is exactly the third fundamental form of $\partial M$. 

Gabai-Meyerhoff-Thurson~\cite{GMT} showed that any homeomorphism between closed hyperbolic 3-manifolds is isotopic to an isometry, refining in dimension 3 the well-known Mostow rigidity~\cite{Mos}. Due to the rigidity and the dual rigidity of convex bodies in $\H^3$, it is interesting to ask if the intrinsic metric or the dual metric of the boundary determine $g$ up to isotopy, and if all metrics satisfying natural necessary conditions are realized for some choice of $g$. For the case of smooth strictly convex boundary this was done by Schlenker in~\cite{Sch4}:

\begin{thm}
\label{sch1}
Let $M$ be admissible. For every smooth Riemannian metric $d$ on $\partial M$ of Gaussian curvature $>-1$ there exists a unique up to isotopy hyperbolic metric $g$ on $M$ such that $\partial M$ is strictly convex in $g$ and the induced intrinsic metric on $\partial M$ is $d$.
\end{thm}

\begin{dfn}
We say that a metric $d$ on $\partial M$ is \emph{large for $M$} if the lengths of all closed geodesics in $(\partial M, d)$ that are contractible in $M$ are greater than $2\pi$.
\end{dfn}

\begin{thm}
\label{sch2}
Let $M$ be admissible. For every smooth large for $M$ Riemannian metric $d$ on $\partial M$ of Gaussian curvature $<1$ there exists a unique up to isotopy hyperbolic metric $g$ on $M$ such that $\partial M$ is strictly convex in $g$ and the induced dual metric on $\partial M$ is $d$.
\end{thm}

It is natural to expect that polyhedral counterparts of Theorem~\ref{sch1} and~\ref{sch2} should hold. Curiously enough, the approach to infinitesimal rigidity (which is the key part of the argument) from~\cite{Sch4} is too much specific to the smooth situation and does not seem to admit a discrete analogue. In the same time, standard methods of proving the infinitesimal rigidity of polyhedra in space-forms use fundamentally the fact that the boundary is the 2-sphere, and become not applicable in more non-trivial topological situations. Hence, we needed to work out a different technique to bridge this gap. For a polyhedral version of Theorem~\ref{sch1} an additional complication arises from the necessity to consider not just hyperbolic 3-manifolds with polyhedral boundary, but a more general class of so-called \emph{bent} hyperbolic 3-manifolds. A good example is provided by \emph{the convex core} of a convex cocompact hyperbolic 3-manifold, which is the inclusion-minimal totally convex subset, see more details in Section~\ref{cocomp}. Its boundary is convex and intrinsically locally isometric to the hyperbolic plane, but it is bent along a geodesic lamination. This lamination may consist of uncountably many disjoint open geodesics winding around each other in a sophisticated manner, see~\cite{Thu, BO} for explanations. It is not known currently, how to establish the infinitesimal rigidity for such manifolds. However, in our very recent work~\cite{Pro3} we prove the existence part of the polyhedral version of Theorem~\ref{sch1} and obtain the uniqueness under mild restrictions on the boundary.

The mentioned effects, however, do not considerably impact the dual problem, as the dual metric degenerates in the mentioned cases (though then it can be replaced by another object, and it is a very interesting task to deal with the rigidity in this setting, see the end of this section). This paper is devoted to a resolution of the dual problem, i.e., to a polyhedral counterpart of Theorem~\ref{sch2}. 

\begin{dfn}
We say that $g$ is a \emph{strictly polyhedral} hyperbolic metric on $M$ if the interior is locally hyperbolic, the boundary is locally isometric to convex polyhedral sets in $\H^3$, and the distance from $\partial M$ to the convex core is positive.
\end{dfn}

Again, for more explanations on convex cores we refer to Section~\ref{cocomp}. Our result is 

\begin{thm}
\label{mt}
Let $M$ be admissible. For each concave large for $M$ spherical cone-metric $d$ on $\partial M$ there exists a unique up to isotopy strictly polyhedral hyperbolic metric $g$ on $M$ such that the induced dual metric on $\partial M$ is $d$.
\end{thm}

This is a generalization of Theorem~\ref{HR} to arbitrary compact hyperbolic 3-manifolds with boundary. We note that the condition of strict convexity in Theorems~\ref{sch1} and~\ref{sch2} implies that the boundary is at positive distance from the convex core. Thus, Theorem~\ref{mt} includes a polyhedral version of Theorem~\ref{sch2}. When the boundary of $M$ touches the convex core, the dual metric degenerates to a pseudo-metric, hence the condition of strict polyhedrality in Theorem~\ref{mt} is necessary to have a well-definite dual metric on $\partial M$. In Section~\ref{polyhsec} we explain that the conditions on $d$ are also necessary (again, this is non-trivial only for the condition of largeness for $M$).

The only previously known case of Theorem~\ref{mt} was established by Schlenker in~\cite{Sch5} for so-called \emph{Fuchsian hyperbolic 3-manifolds of the first kind} with convex polyhedral boundary. They are homeomorphic to $S \times [-1, 1]$, where $S$ is a closed surface of genus $>1$, and contain an embedded totally geodesic surface isotopic to $S \times \{0\}$. See also a paper of Fillastre~\cite{Fil}. We note that our Theorem~\ref{mt} answers positively to Question~2 in~\cite{Fil}. In the Fuchsian case of the first kind Fillastre also established in~\cite{Fil2} a polyhedral counterpart to Theorem~\ref{sch1}. Similar results for the Minkowski and anti-de Sitter 3-spaces were obtained in~\cite{Fil}.

A particularly interesting case of Theorem~\ref{mt} is given by so-called \emph{quasi-Fuchsian hyperbolic 3-manifolds} with boundary. Those are also homeomorphic to $S \times [-1, 1]$, but the condition on the embedded geodesic surface is dropped. There is a pair of domains in $\dS^3$, on which the holonomy map of a quasi-Fuchsian manifold acts freely and properly discontinuously.
The quotient of every of them is known in Lorentzian geometry as a \emph{globally hyperbolic maximal Cauchy compact} (GHMC for short) (2+1)-spacetime of constant curvature 1. See, e.g., the fundamental paper of Mess~\cite{Mes} for more details and a paper of Scannell~\cite{Sca}. This case of Theorem~\ref{mt} can be interpreted as a convex polyhedral realization and rigidity result in such spacetimes.

Our proof of Theorem~\ref{mt} follows the standard continuity approach popularized by Alexandrov. We consider the space $\mathcal P(M, V)$ of isotopy classes of strictly polyhedral metrics on $M$ with marked faces, and the space $\mathcal D_M(\partial M, V)$ of isotopy classes of concave large for $M$ spherical cone-metrics on $\partial M$ with marked vertices. We endow these spaces with appropriate topologies and consider the dual metric map $\mathcal I_V: \mathcal P(M, V) \rightarrow \mathcal D_M(\partial M, V)$. We show that $\mathcal I_V$ is a $C^1$-diffeomorphism. 

The proof is decomposed into three pieces: establishing differentiability properties of $\mathcal I_V$ and its infinitesimal rigidity, showing its properness, and studying connectivity properties of the spaces $\mathcal P(M, V)$ and $\mathcal D_M(\partial M, V)$. The last item was basically done before us: it is a generalization of the connectivity results from~\cite{HR}, as was mentioned by Schlenker in~\cite{Sch5}. Thus, we mostly focus on the infinitesimal rigidity and the properness. Note also that the connectivity results on the spaces above are obtained only in some weak sense, and a special topological argument is required to finish the proof.

We deduce the infinitesimal rigidity via the doubling construction from the works on the infinitesimal rigidity of hyperbolic cone-3-manifolds due to Weiss and Montcouquiol-Weiss~\cite{Wei, Wei2, MW}. Note that in the presence of vertices of valence $>3$ the infinitesimal rigidity in our problem is not the same as the infinitesimal rigidity with respect to the dihedral angles, which follows from the aforementioned works rather directly, and we need to extract the former with some care. Hyperbolic cone-3-manifolds have interesting applications to various geometric problems, see, e.g.,~\cite{BLP, Bro, HK2, Pro2, Pro}. The proof of properness is partially inspired by Schlenker's works~\cite{Sch2, Sch3}, which deal with Alexandrov-type problems for polyhedra in the union of $\H^3$ and $\dS^3$. 

An interesting issue with the continuity method is that it gives no construction of a hyperbolic metric $g$ inducing the dual metric $d$. This might be resolved by a variational approach. A variational proof of Alexandrov's Theorem was given by Volkov, a student of Alexandrov, in his thesis, published only recently as~\cite{Vol}. A refinement of Volkov's approach was given in~\cite{BI} by Bobenko--Izmestiev. It was applied to similar problems for polyhedral hyperbolic 3-manifolds, e.g., in~\cite{FI, Pro}. A particularly relevant for us reference is~\cite{FI2} where the problem of induced dual metrics on the so-called \emph{hyperbolic 3-cusps} with convex polyhedral boundary was resolved by variational techniques. It would be very interesting to obtain a variational proof of Theorem~\ref{mt} or a proof of a polyhedral analogue of Theorem~\ref{sch1}, but at the moment there are difficulties with attempts to implement such intentions. 

Up to now we discussed in parallel questions for polyhedral and for smooth boundaries. It is very natural to ask if there is a common generalization unifying these two settings. This was famously initiated by Alexandrov in his work on general metrics on surfaces with curvature bounded from below~\cite{Ale}. He particularly proved a realization theorem, obtaining a complete intrinsic classification of the induced path metrics on the boundaries of convex bodies in $\E^3$, $\H^3$ or $\S^3$. The rigidity question for general convex bodies was settled by Pogorelov~\cite{Pog} and is known as notoriously difficult. In recent years there was a stream of work showing general realization results on hyperbolic 3-manifolds, e.g.,~\cite{Slu, FIV}. The only known general rigidity result for manifolds topologically more complicated than the 3-ball (for Fuchsian manifolds of the first kind) was done by the author in~\cite{Pro2}. The corresponding results for the dual metrics are still not known. An important particular case, again, is given by the convex cores of convex cocompact hyperbolic 3-manifolds. Their dual metric is somewhat degenerate, but can be replaced with another object called \emph{the pleating lamination} measuring how their boundary is bent. Thurston conjectured that convex cores are uniquely determined by their pleating lamination as well as by the induced boundary metric. The realization part of the former question was shown in~\cite{BO} and~\cite{Lec} by Bonahon--Otal and Lecuire, while for the latter it is, e.g., a particular corollary of the work~\cite{Lab} by Labourie, where another proof of the realization part of Theorem~\ref{sch1} was given. The rigidity part for pleating laminations was recently resolved by Dular--Schlenker in~\cite{DS}, while for the induced metric it remains an open problem, with some progress obtained in~\cite{Pro3}. This is reminiscent to the famous Ahlfors--Bers Theorem (see Theorem~\ref{ab}), which could be stated as that convex cocompact hyperbolic metrics on 3-manifolds are determined up to isotopy by the conformal structure at infinity.

\vskip+0.2cm

\textbf{Acknowledgments.} We thank Fran\c{c}ois Fillastre, Ivan Izmestiev, Jean-Marc Schlenker and the anonymous referees for helpful comments. We are also grateful to our mentor, Michael Eichmair, for his support.

\section{Preliminaries}

\subsection{Hyperbolic and de Sitter 3-spaces and the duality}
\label{minkow}

For some expositions on the duality between hyperbolic and de Sitter spaces we refer to \cite{HR, Sch2, FS}.

By $\langle .,. \rangle_M$ we denote the \emph{Minkowski scalar product} on $\R^4$, i.e.,
$$\langle x, y \rangle_M := -x_0y_0+x_1y_1+x_2y_2+x_3y_3.$$
The hyperbolic 3-space $\H^3$ can be identified then with the upper sheet of the two-sheeted hyperboloid
$$\H^3=\{x \in \R^4: \langle x,~x \rangle_M = -1, x_0>0\}$$
endowed with the induced metric. The de Sitter 3-space $\dS^3$ can be identified with the one-sheeted hyperboloid
$$\dS^3=\{x \in \R^4: \langle x, x \rangle_M =1\}$$
endowed with the induced metric. It is a Lorentzian 3-manifold of constant curvature 1 homeomorphic to $S^2 \times \R$. Let us say that a subset of $\H^3$ or $\dS^3$ is \emph{properly convex} when it is proper, has non-empty interior and is the intersection of closed half-spaces (the latter is equivalent to be closed and convex in $\H^3$, but the convexity is more subtle in $\dS^3$).

As the hyperbolic 3-space $\H^3$ is compactified by adding the boundary at infinity $\partial_{\infty} \H^3$, the de Sitter 3-space is compactified similarly by adding the boundary at infinity $\partial_{\infty} \dS^3$ consisting of two connected components, $\partial_{\infty}^+ \dS^3$ and $\partial_{\infty}^- \dS^3$, corresponding to the two parts of the light cone of the origin in $\R^4$. The component $\partial_{\infty}^+ \dS^3$ is naturally identified with $\partial_{\infty} \H^3$. We consider $\dS^3$ time-oriented with the future-directions pointing towards $\pt_\infty^+ \dS^3$.

There is a well-known duality between oriented planes in $\H^3$ and points in $\dS^3$, and a duality between space-like planes in $\dS^3$ and points in $\H^3$. Indeed, each plane in $\H^3$ is the intersection of $\H^3$ with a linear hyperplane in $\R^4$ passing through the origin. Its unit normal in the positive direction belongs to $\dS^3$. Conversely, for any point in $\dS^3$ the orthogonal hyperplane in $\R^4$ intersects $\H^3$ and determines an oriented geodesic plane there. Similarly, one produces a correspondence between space-like planes in $\dS^3$ and points in $\H^3$. Note that we consider the planes oriented only in $\H^3$. To avoid this asymmetry, one either needs to pass to the projective space (and consider the respective quotient of $\dS^3$), or add to the consideration the mirror copy of $\H^3$, however, for our purposes we prefer not to do this. 

If $C \subset \H^3$ is properly convex, then the dual points in $\dS^3$ to all outward-oriented planes that do not intersect the interior of $C$ form the dual properly convex set $C^* \subset \dS^3$. We say that $C^*$ is \emph{the Gauss dual} of $C$. By \emph{the Gauss map} we mean the multivalued map $\partial C \rightrightarrows \partial C^*$ sending a point $p \in \partial C$ to the set of the outward normals to all supporting planes to $C$ at $p$. It is easy to see that no supporting plane to $C^*$ is time-like, and that $\partial_{\infty}^+ \dS^3$ belongs to the convex side of $C^*$. We will call a properly convex set in $\dS^3$ satisfying these two conditions \emph{future-convex}. Equivalently, a future-convex set is a properly convex set that contains the future cone of every point, or a properly convex set that is the intersection of closed future half-spaces bounded by space-like planes. Conversely, for any future-convex set $C \subset \dS^3$ we similarly define its dual convex set $C^* \subset \H^3$. This duality is polar, i.e., $C^{**}=C$.

\subsection{Cone-metrics}
\label{conemetsec}

Let $S$ be a closed oriented surface and $V \subset S$ be a finite set of points.

\begin{dfn}
A \emph{triangulation} $\mathcal T$ of $(S, V)$ is a collection of simple disjoint paths with endpoints in $V$ that cut $S$ into triangles. Here a triangle is a 2-disk with three marked points at the boundary (possibly coinciding as points of $V$). Two triangulations with the same vertex set $V$ are \emph{equivalent} if they are isotopic by an isotopy fixing $V$.
\end{dfn}

We denote by $E(\mathcal T)$ the set of edges of $\mathcal T$.



Let $d$ be a cone-metric on $S$ (Definition~\ref{conedef}) and $V \subset S$ be a finite set containing the singular set $V(d)$ of $d$. In this situation we say that $d$ is a cone-metric on $(S, V)$.

\begin{dfn}
A \emph{geodesic triangulation} of $(S, V, d)$ is a triangulation of $(S, V)$ such that all edges are geodesics in $d$. 
\end{dfn}

\begin{dfn}
Let $\mathcal T$ be a triangulation of $(S, V)$ and $d$ be a cone-metric on $(S, V)$. We say that $d$ is \emph{$\mc T$-triangulable} if there is a geodesic triangulation of $(S, V, d)$ equivalent to $\mathcal T$. If $d$ is spherical, then we require additionally that each triangle is convex, i.e., isometric to a triangle in an open hemisphere.
\end{dfn}

We remark that not all triples $(S, V, d)$, when $d$ is a spherical cone-metric, admit a geodesic triangulation even without the condition on the convexity of triangles. An example is the standard sphere $\S^2$ with $V$ being three points contained in an open hemisphere. However, for concave large spherical cone-metrics there are no such obstructions.


\begin{lm}
\label{triang}
For each concave large spherical cone-metric $d$ on $(S, V)$ there exists a geodesic triangulation $\mathcal T$ of $(S, V, d)$ into convex triangles.
\end{lm}

This was basically shown by Rivin in~\cite{Riv} although he does not formulate the convexity condition on the triangles, but it follows from his proof. For surfaces of genus $>1$, which is the case we need, this also follows directly from the Fuchsian version of Theorem~\ref{sch2} due to Schlenker~\cite{Sch5}. It says that each concave large spherical cone-metric on $S$ can be realized as the dual metric on the boundary of a Fuchsian manifold of the first kind. When we consider the respective dual embedding of the universal cover to $\dS^3$, we see that any invariant geodesic triangulation refining the face decomposition projects to a geodesic triangulation of the metric into convex triangles.

We claim that if a concave large spherical cone-metric $d$ on $(S, V)$ is $\mc T$-triangulable, then a geodesic triangulation into convex triangles equivalent to $\mc T$ is unique. Indeed, if between two points on the standard 2-sphere there exists a geodesic segment of length $<2\pi$ (i.e., the points are not diametrically opposite), then this segment is unique. Hence, if two geodesic triangulations of $(S, V, d)$ into convex triangles are isotopic, then the lengths of all edges are equal. It follows that there exists a self-isometry of $(S, V, d)$ isotopic to the identity by an isotopy fixing $V$. Consider the Fuchsian realization of $(S, V, d)$ given by~\cite{Sch5}. Since the Fuchsian realization is unique, the self-isometry of $d$ extends to a self-isometry of the Fuchsian realization. Every Fuchsian manifold of the first kind contains a unique totally geodesic copy of $S$, which then also admits a self-isometry isotopic to the identity. It is, however, easy to see that no closed hyperbolic surface admits such an isometry.


For a map $f: (S_1, d_1) \rightarrow (S_2, d_2)$ between two compact metric spaces we define its \emph{distortion} ${\rm dist}(f)$ as
$${\rm dist}(f):=\sup_{p, q\in S_1, p \neq q}\left|\ln\frac{d_2(f(p), f(q))}{d_1(p,q)}\right|.$$
We say that a sequence of metrics $\{d_i\}$ on a surface $S$ converges to a metric $d$ \emph{in the Lipschitz sense} if the identity maps $(S, d_i) \rightarrow (S, d)$ have distortions $\e_i \rightarrow 0$. We remark that some authors allow the action by homeomorphisms in the definitions of the Lipschitz convergence, e.g., in~\cite[Section 6]{HR}. In a more general setting the Lipschitz convergence may be defined between non-homeomorphic metric spaces, see~\cite[Section 7.2]{BBI}.
We denote by $\mathfrak D(S, V)$ the set of spherical cone-metrics on $(S, V)$ that are $\mc T$-triangulable for at least one triangulation. We endow it with the topology of Lipschitz convergence.

We will need two equivalence relations on $\mf D(S, V)$. Denote by $H_0(S, V)$ the group of self-homeomorphisms of $S$ fixing $V$ and isotopic to identity, and denote by $H^\sharp_0(S, V)$ its normal subgroup consisting of those homeomorphisms that are isotopic to the identity by an isotopy fixing $V$. The quotient $B_0(S, V):=H_0(S, V)/H^\sharp_0(S, V)$ is the \emph{pure braid group} of $(S, V)$. The group $H_0(S, V)$ acts on $\mf D(S, V)$ and we denote the quotient by $\mc D(S, V)$, while we denote the quotient by $H^\sharp_0(S, V)$ by $\mc D^\sharp(S, V)$. It follows that $\mc D(S, V)$ is the quotient of $\mc D^\sharp(S, V)$ by $B_0(S, V)$.  When the context is clear, we may abuse the notation and write $d \in \mathcal D(S, V)$, when $d$ is a metric representing a class belonging to $\mathcal D(S, V)$.
We will mostly deal with the space $\mc D(S, V)$, though in the last stages of the proof, for technical reasons we will have to employ the space $\mc D^\sharp(S, V)$.


For a triangulation $\mathcal T$ of $(S, V)$ we denote by $\mf D(S, \mathcal T) \subset \mf D(S, V)$ the subset of $\mc T$-triangulable metrics. It is $H^\sharp_0(S, V)$-invariant, and we denote by $\mathcal D^\sharp(S, \mathcal T) \subset \mathcal D^\sharp(S, V)$ its quotient. 
The set $\mathcal D^\sharp(S, \mathcal T)$ can be considered as an open polyhedron in $\R^{E(\mathcal T)}$ defined by the strict triangle inequalities and the conditions on edge lengths to be in $(0, \pi)$. The charts $\mathcal D^\sharp(S, \mathcal T) \ra \R^{E(\mc T)}$ endow $\mathcal D^\sharp(S, V)$ with the structure of a real-analytic manifold of dimension $3(n-k)$ where $n:=|V|$ and $k:=\chi(S)$. It is easy to see that this topology coincides with the quotient topology induced from $\mf D(S, V)$. The group $B_0(S, V)$ acts on $\mc D^\sharp(S, V)$ freely and properly discontinuously, thereby $\mc D(S, V)$ is also a real-analytic manifold of dimension $3(n-k)$.

Pick $d \in \mf D(S, \mc T)$. There is a canonical section through $d$ of the forgetful map $\mf D(S, \mc T) \ra \mc D^\sharp(S, \mc T)$. Indeed, for any two convex spherical triangles $T_1$ and $T_2$ with marked vertices there exists a canonical ``affine'' homeomorphism between them preserving the marking. To see this, develop $T_1$ and $T_2$ on the standard sphere $\S^2$ embedded in Euclidean 3-space $\R^3$. Project $T_1$ radially to the Euclidean triangle subtending it, send it by the unique linear map to the affine triangle subtending $T_2$ and project it radially to $T_2$. This defines the desired map. For every class $d \in \mf D(S, \mc T)$ such affine maps produce the canonical section of $\mf D(S, \mc T) \ra \mc D^\sharp(S, \mc T)$ through $d$. Acting by $H^\sharp_0(S, V)$ on the metrics of this section, one can see that $\mf D(S, \mc T) \cong \mc D^\sharp(S, \mc T) \times H^\sharp_0(S, V)$. This produces the structure of a principal bundle on $\mf D(S, V) \ra \mc D^\sharp(S, V)$ with the structure group $H^\sharp_0(S, V)$. Since the latter is known to be contractible~\cite{Ham}, this bundle is trivial.

We will need the following simple lemma:

\begin{lm}
\label{fintriang}
Let $d \in \mf D(S, V)$. There exists a neighborhood $U \ni d$ in $\mf D(S, V)$ such that there are only finitely many classes of triangulations $\mc T$, for which there exists a $\mc T$-triangulable metric in $U$.
\end{lm}

\begin{proof}
We show that for any $v, w \in V$ there are only finitely many isotopy classes on $S \backslash V$ of paths between $v$ and $w$ containing a geodesic path of length $\leq \pi$ in $d$. Indeed, suppose the converse. 
The set of all geodesic paths from $v$ to $w$ of length $\leq \pi$ is uniformly Lipschitz. Thus, by the Arzel\'a--Ascoli Theorem it is compact. Hence, if there are infinitely many such classes, we can take a geodesic representative of length $\leq \pi$ from each, and extract a converging subsequence. However, every geodesic path of length $< \pi$ between $v$ and $w$ is isolated, i.e., it has a neighborhood where there are no other geodesic paths on $S\backslash V$ between its endpoints, and every geodesic path of length $\pi$ is isolated from geodesic representatives of other classes on $S\backslash V$. Thus, no path can appear in the limit.

This argument implies that there are finitely many equivalence classes of triangulations that have a geodesic representative in $d$ with all edges $\leq \pi$. Thus, there are finitely many triangulations $\mathcal T$ such that $d$ belongs to the closure of $\mf D(S, \mathcal T)$. This finishes the proof.
\end{proof}


Denote by $\mf D_c(S, V) \subset \mf D(S, V)$ the subset of concave spherical cone-metrics on $(S, V)$ with $V(d)=V$ and by $\ol{\mf D}_c(S, V)\subset \mf D(S, V)$ denote just the subset of concave spherical cone-metrics on $(S, V)$. The latter is the topological closure of the former in $\mf D(S, V)$. The same notation applies to the respective subsets of $\mc D^\sharp(S, V)$ and $\mc D(S, V)$, and also when $V$ is replaced by $\mc T$.

Now let $M$ be an admissible 3-manifold, $V \subset \partial M$ be a finite set of points with at least one point at each boundary component. We denote by $\mf D_M(\pt M, V)$ the subset of concave large for $M$ spherical cone-metrics on $(\pt M, V)$ with $V(d)=V$. We denote by $\ol{\mf D}_M(\pt M, V) \subset \mf D(\pt M, V)$ just the subset of concave large for $M$ spherical cone-metrics on $(\pt M, V)$. We note that here our notation has a downside as $\ol{\mf D}_M(\pt M, V)$ is not the topological closure of $\mf D_M(\pt M, V)$ because the topological closure also contains metrics with closed contractible in $M$ geodesics of length $2\pi$. The same notation applies to the respective subsets of $\mc D^\sharp(\pt M, V)$ and $\mc D(\pt M, V)$, and also when $V$ is replaced by $\mc T$. We will show that $\mf D_M(\partial M, V)$ is an open subset of $\mf D(\partial M, V)$. Clearly, the concavity condition is open, so we need to check the condition of largeness for $M$. For a metric $d$ on $\partial M$ denote by ${\rm ml}_M(d)$ the infimum of lengths of all closed geodesics in $(\partial M, d)$ that are contractible in $M$.

We denote by $\S^1_l$ the metric circle of length $l$. A \emph{closed global geodesic} in a metric space $(S, d)$ is the image of an isometric map $\S^1_l \rightarrow (S, d)$ (here it is crucial that it preserves distances in $\S^1_l$ globally). A \emph{closed $\e$-quasi-geodesic} is the image of $S^1_l$ under a map $S^1_l \rightarrow (S, d)$ with distortion at most $\e$. The following consequence of the Arzel\'a--Ascoli Theorem was shown in~\cite[Theorem 6.1]{HR}:

\begin{lm}
\label{geodconv}
Let $\{d_i\}$ be a sequence of metrics on $S$ converging in the Lipschitz sense to a metric $d$. Let $\{\gamma_i \subset S\}$ be closed $\e_i$-quasi-geodesics in $d_i$ with $\e_i \rightarrow 0$ and the lengths $l_i$ of $\gamma_i$ converging to $l>0$. Then up to extracting a subsequence, $\gamma_i$ converge to a closed global geodesic $\gamma \subset S$ in $d$ of length $l$.
\end{lm}

From this we can deduce

\begin{crl}
Let $\{d_i\}$ be a sequence of metrics on $\partial M$ converging in the Lipschitz sence to a metric $d$. Then
$${\rm ml}_M(d) \leq \liminf_{i \rightarrow \infty} {\rm ml}_M(d_i).$$
\end{crl}

\begin{proof}
When a sequence of closed curves on $\partial M$ converges, their free homotopy classes stabilize. Hence, the limit of a sequence of curves in $\partial M$ that are contractible in $M$ is also contractible in $M$. Next, for each closed geodesic in $(\partial M, d_i)$ of length $l$ that is contractible in $M$ there exists a closed global geodesic of length at most $l$ that is contractible in $M$. Indeed, this is because there is a correspondence between closed geodesics in $(\partial M, d_i)$ contractible in $M$ and closed geodesics in $(\partial \tilde M, d_i)$, where $\tilde M$ is the universal cover of $M$. Thus, by Lemma~\ref{geodconv} the infimum ${\rm ml}_M(d_i)$ is realized by a closed global geodesic $\gamma_i$ contractible in $M$. Applying Lemma~\ref{geodconv} again, we see that a subsequence of $\gamma_i$ converges to a closed global geodesic in $(\partial M, d)$ contractible in $M$ of length $\liminf_{i \rightarrow \infty} {\rm ml}_M(d_i)$. 
\end{proof}

\begin{crl}
\label{semicont}
The function ${\rm ml}_M$ is lower-semicontinuous in the Lipschitz topology of metrics on $\partial M$, i.e., for each $l \in \R$ the set of metrics $d$ on $\partial M$ satisfying ${\rm ml}_M(d)>l$ is open. In particular, $\mf D_M(\partial M, V)$ is an open subset of $\mf D(\partial M, V)$.
\end{crl}

\subsection{Convex cocompact hyperbolic manifolds}
\label{cocomp}

Let $g$ be a hyperbolic metric on $M$ such that $\partial M$ is locally convex. Recall that we denote by $G$ the group of orientation-preserving isometries of $\H^3$. It is naturally isomorphic to ${\rm PSL}(2, \C)$.


A developing map embeds the universal cover $(\tilde M, g)$ as a convex set in $\H^3$ (see~\cite[Section 1.4]{CEG}) and the respective holonomy map is a discrete and faithful representation $\rho: \pi_1(M) \rightarrow G$. The quotient $N$ of $\H^3$ by $\rho$ is a complete hyperbolic 3-manifold, to which $M$ naturally embeds inducing an isomorphism of the fundamental groups. Denote this embedding by $\iota: M \hookrightarrow N$. It is easy to see that $N$ is homeomorphic to $\inter(M)$ by a map $\kappa$ such that $\iota\circ\kappa$ is homotopic to the identity.


For the rest of the paper, however, we will use a convention that $N$ is a smooth 3-manifold without a hyperbolic structure, but homeomorphic to $\inter(M)$. We will consider such a homeomorphism fixed for the whole paper, so, particularly, there is a fixed 1-1 correspondence between ends of $N$ and components of $\pt M$. The construction above endows $N$ with a convex cocompact hyperbolic metric defined up to isotopy, which we denote by $\ol g$, and with an isometric embedding $\iota:(M, g)\ra (N, \ol g)$.
We will denote the pair $(N, \overline g)$ rather by $N(\overline g)$, so $N(\overline g)$ is a hyperbolic 3-manifold, while $N$ is the underlying smooth 3-manifold. We denote by $\partial_{\infty} N(\overline g)$ its boundary at infinity, which is homeomorphic to $\partial M$ and is equipped with a conformal structure. We denote by $M(g)$ the $\iota$-image of $(M, g)$ in $N(\overline g)$. Generally we denote by $\overline g$ the convex cocompact hyperbolic metric on $N$ obtained from a metric $g$ as above. However, sometimes we will abuse the notation, and will denote by $\overline g$ a convex cocompact hyperbolic metric on $N$ in the absence of a metric $g$.


Abusing the notation, we denote the extension of $\overline g$ to the universal cover $\tilde N$ still by $\overline g$, and the hyperbolic 3-manifold $(\tilde N, \overline g)$ by $\tilde N(\overline g)$. We will frequently identify $\tilde N(\overline g)$ with $\H^3$ by an arbitrary developing map (recall that a developing map is defined up to a global isometry of $\H^3$).
We denote by $\tilde M(g)$ the respective embedding of the universal cover of $M(g)$ as a convex set. 
Let $\rho_{\overline g}: \pi_1(M) \rightarrow G$ be the corresponding holonomy representation. We denote by $\Lambda(\overline g)=\Lambda(\rho_{\ol g})$ its limit set and by $\tilde C(\overline g)=\tilde C(\rho_{\overline g})$ we denote the convex hull of $\Lambda(\overline g)$. If $\Lambda(\overline g)$ is contained in a circle, we say that $\overline g$ and $\rho_{\overline g}$ are \emph{Fuchsian}. They are \emph{Fuchsian of the first kind} if $\Lambda(\overline g)$ is a circle, and \emph{of the second kind} otherwise. The set $\tilde C(\overline g)$ is $\rho_{\overline g}$-invariant and projects to the \emph{convex core} $C(\overline g) \subset N(\ol g)$. The latter is totally convex, and is homeomorphic to $M$, except in the Fuchsian case. We will refer to the connected components of $N(\ol g) \backslash C(\ol g)$ and of $\tilde N(\ol g) \backslash \tilde C(\ol g)$ as to the \emph{geometric ends} of $N(\ol g)$ and $\tilde N(\ol g)$. Since the $\rho_{\overline g}$-orbit of any point of $\H^3$ accumulates at $\Lambda(\overline g)$, $C(\overline g)$ is contained in any totally convex subset of $N(\overline g)$. It is also well-known that

\begin{lm}
The induced path metric on $\partial C(\overline g)$ is hyperbolic.
\end{lm}

See, e.g.,~\cite{EM} or~\cite{Smi}. Thus, a connected component of $\partial C(\overline g)$ is an embedded hyperbolic surface, which, however, may not be  totally geodesic as it may be bent along a geodesic lamination. Here we mean by $\partial C(\overline g)$ the boundary in $N(\overline g)$, so in the Fuchsian case, when $C(\overline g)$ is 2-dimensional, we mean $\partial C(\overline g)=C(\overline g)$. In the Fuchsian case of the first kind $C(\overline g)$ is an embedded closed totally geodesic surface, and $M$ is an interval bundle over this surface. In the Fuchsian case of second kind $C(\overline g)$ is a totally geodesic surface with geodesic relative boundary, and $M$ is a handlebody, so it has compressible boundary. 

%
%
%

Denote by $\mc{CH}(N)$ the space of convex cocompact hyperbolic metrics on $N$, where we consider two metrics equivalent if they are related by a self-homeomorphism of $M$ isotopic to the identity on $M$, where we use $N \cong \inter(M)$. The rest of this section is devoted to the discussion of a topology on this space. While it seems to be folklore, we did not find in the literature a reference suitable to our needs. Thereby, we devote some time to extract it from the existing literature. We will frequently abuse the notation and write $\overline g \in \mathcal{CH}(N)$ meaning that $\overline g$ is a concrete convex cocompact hyperbolic metric on $N$. Note that one can also consider metrics up to isotopy on $N$, which a priori can lead to a stricter equivalence. However, this is not the case, see Remark~\ref{mcg}, but we will not use it. 

It turns out to be easier to describe the topology on the space $\mc{CH}_h(N)$ of convex cocompact hyperbolic metrics on $N$ up to homotopy, i.e., up to isometry homotopic to the identity. Hence, first we accomplish this, and then we discuss the difference between $\mc{CH}(N)$ and $\mc{CH}_h(N)$.

We denote by $\mathcal R(\pi_1(M), G)$ the space of representations of $\pi_1(M)$ in $G$. The group $G$ acts on $\mathcal R(\pi_1(M), G)$ by conjugation, and we denote the resulting quotient by $\mathcal X(\pi_1 (M), G)$. Since $G\cong {\rm PSL}(2, \C)$ is an algebraic group over $\C$, the space $\mathcal R(\pi_1(M), G)$ can be viewed as a complex algebraic variety, and we consider it with the induced topology. We endow the space $\mathcal X(\pi_1 (M), G)$ with the quotient topology. Note that $\mathcal X(\pi_1 (M), G)$ is not even a Hausdorff space, and should not be confused with \emph{the character variety}, which is a more accurate quotient that can be constructed by means of algebraic geometry. However, we are interested only in the quotient of an invariant open smooth subset of $\mathcal R(\pi_1(M), G)$, corresponding to the holonomies of convex cocompact metrics, over which the action behaves well. We will be anyway calling the elements of $\mathcal X(\pi_1 (M), G)$ \emph{characters}.

Denote by $\mc{HM}_h(N)$ the space of complete hyperbolic metrics on $N$ up to homotopy, so $\mc{CH}_h(N) \subset \mc{HM}_h(N)$.
The character of holonomy sends injectively $\mathcal{HM}_h(N)$ to $\mathcal X(\pi_1(M), G)$. Indeed, if two hyperbolic metrics $\overline g_1$, $\overline g_2$ have the same character, then we have an orientation-preserving isometry between $N(\overline g_1)$ and $N(\overline g_2)$ inducing the identity map on the fundamental groups. Since $M$ is aspherical, $\overline g_1$ and $\overline g_2$ are homotopic. We endow $\mathcal{HM}_h(N)$ with the topology pull-backed from $\mathcal X(\pi_1(M), G)$. Note that we can endow $\mathcal{HM}_h(N)$ with a more natural topology, which comes from $C^{\infty}$-convergence on compact subsets for developing maps. Then the so-called Ehresmann-Thurston Theorem implies that these topologies are the same. (The non-trivial part is that the convergence of holonomies implies, up to homotopy, the convergence of the developing maps. It follows, e.g., from applying ~\cite[Theorem 1.7.1]{CEG} to compact totally convex subsets  exhausting $N$.)

Let $\overline g \in \mathcal{CH}_h(N)$ and $\rho_{\overline g}$ be its holonomy. It follows from Marden's Isomorphism Theorem~\cite[Section 8]{Mar} that a convex cocompact representation $\rho \in \mathcal R(\pi_1(M), G)$ is a holonomy representation of a convex cocompact metric on $N$ if and only if it is related to $\rho_{\overline g}$ by a quasi-conformal deformation. Thus, the set of holonomy representations of metrics from $\mathcal{CH}_h(N)$ can be identified with the set $\mathcal{QC}(\rho_{\overline g})$ of quasi-conformal deformations of $\rho_{\overline g}$. In particular, it is connected. Marden's Stability Theorem~\cite[Section 10]{Mar} implies also that it is an open subset of the smooth part of $\mathcal R(\pi_1(M), G)$. The set $\mathcal{QC}(\rho_{\overline g})$ stays invariant under the action of $G$ by conjugation, and $G$ acts freely and smoothly on it. Moreover,~\cite[Theorem 10.8]{Mar} implies that this action is proper, so the quotient is a complex manifold. Other approaches to the properness can be seen in~\cite[Proposition 1.1]{JM} or in~\cite[Lemma 6.24]{Wei}. 
We identify $\mathcal{CH}_h(N)$ with this quotient. We refer to the expositions in~\cite[Chapter 5]{Mar2} and~\cite[Section 7]{CM}.

Now we discuss the difference between equivalences of metrics up to homotopy and up to isotopy. We denote by $\MCG(M)$ the \emph{mapping class group} of $M$, i.e., the group of isotopy classes of orientation-preserving self-homeomorphisms of $M$. We denote by $\MCG_h(M) \subset \MCG(M)$ the subgroup of isotopy classes of orientation-preserving self-homeomorphisms that are homotopic to the identity. There is a natural homomorphism $\MCG(M) \rightarrow {\rm Out}(\pi_1(M))$, which sends a homeomorphism class to the induced outer automorphism of $\pi_1(M)$. Clearly, $\MCG_h(M)$ belongs to the kernel of this homomorphism. 

The work of Waldhausen~\cite{Wal} establishes that when the boundary of $M$ is incompressible, the group $\MCG_h(M)$ is trivial. However, this is no longer the case when the boundary is compressible. Let $\gamma \subset \partial M$ be an oriented simple closed curve that is contractible in $M$, and $D \subset M$ be an embedded 2-disk contracting $\gamma$, which exists thanks to the Dehn Lemma. Do a Dehn twist around $\gamma$. Let $D \times [-1, 1] \subset M$ be a small neighborhood of $D$ such that $D \times \{0\}=D$ and $\partial D \times [-1, 1]$ is an annulus around $\gamma$ supporting the Dehn twist. The Dehn twist naturally extends to $D \times [-1, 1]$. Extending it by identity outside of $D \times [-1, 1]$, we obtain a self-homeomorphism of $M$~called a \emph{twist} of $M$. One can see that it is homotopic to the identity, but not isotopic. McCullough and Miller showed in~\cite{MM2} that the kernel of $\MCG(M) \rightarrow {\rm Out}(\pi_1(M))$ is generated by twists. Thereby, $\MCG_h(M)$ is equal to the kernel and is generated by twists. Let $S_j$ be a connected component of $\partial M$. We denote by $\MCG_{h, j}(M)$ the subgroup of $\MCG_h(M)$ generated by twists along curves belonging to $S_j$. As compressing disks of disjoint simple closed curves in $\partial M$ can be made disjoint, $\MCG_h(M)$ is the direct product of all $\MCG_{h, j}(M)$. 

\begin{rmk}
\label{mcg}
One can consider a similarly-defined group $\MCG_h(N)$. We have a morphism $\MCG_h(M) \ra \MCG_h(N)$. One can show that this is an isomorphism. We will not rely on it, hence we do not go into the details of it.
\end{rmk}

We denote by $\mathcal{TS}(\partial M)$ the Teichm\"uller space of $\partial M$. It is the direct product of the Teichm\"uller spaces of all connected components of $\partial M$. The group $\MCG_h(M)$ acts on $\mathcal {TS}(\partial M)$ and the action splits on actions of $\MCG_{h, j}(M)$ on the respective components of $\mathcal{TS}(\partial M)$. As the action of the mapping class group of $\partial M$ on $\mathcal {TS}(\partial M)$ is propely discontinuous (see, e.g.,~\cite[Theorem 12.2]{FM}), so is the action of $\MCG_h(M)$. In~\cite[Section 3.1]{Mar3} Marden shows that this action is free. We denote by $\mathcal{TS}_h(\partial M, M)$ the quotient of $\mathcal {TS}(\partial M)$ over this action. 

In~\cite[Section 3.1]{Mar3} Marden states a version of the Ahlfors--Bers Theorem (based also on the works of Kra and Maskit), which says that if $\overline g$ is a convex cocompact hyperbolic metric on $N$, then the action of quasi-conformal deformations of $\rho_{\overline g}$ on $\partial_{\infty} \H^3$ defines a biholomorphic map from the quotient of $\mathcal{QC}(\rho_{\overline g})$ by conjugation of $G$ to $\mathcal{TS}_h(\partial M, M)$. We refer to~\cite{Mar3} and references therein. Additional expositions can be found in~\cite[Section 5.1]{Mar2} or \cite[Section 7]{CM}. Due to the discussion above, we can state this result as

\begin{thm}
\label{abh}
The map $\mathcal{CH}_h(N) \rightarrow \mathcal{TS}_h(\partial M, M)$ sending a class represented by a metric $\overline g$ to the class of its conformal structure at infinity $\partial_{\infty}N(\overline g)$ is a biholomorphic map.
\end{thm}

Finally, we return to $\mathcal{CH}(N)$. We denote by $\mathcal {AB}: \mathcal{CH}(N) \rightarrow \mathcal{TS}(\partial M)$ a similar map defined as above. Clearly, it is equivariant with respect to the action of $\MCG_h(M)$. Together with Theorem~\ref{abh} it implies

\begin{thm}
\label{ab}
The map $\mathcal{AB}$ is bijective.
\end{thm}

We endow $\mathcal{CH}(N)$ with the induced topology. It becomes a connected complex manifold of real dimension $-3k$, where $k=\chi(\partial M)$, and the natural projection $\mathcal{CH}(N)\rightarrow \mathcal{CH}_h(N)$ becomes a covering map.

\subsection{Polyhedral hyperbolic manifolds}
\label{polyhsec}

\begin{dfn}
We say that a hyperbolic metric $g$ on $M$ is \emph{polyhedral} if each point $p \in \partial M(g)$ has a neighborhood isometric to a convex polyhedral set in $\H^3$. It is \emph{strictly polyhedral} if, in addition, the distance between the boundary and the convex core is positive.
\end{dfn}

For $p \in \partial M$ the \emph{spherical link} of $p$ in $M(g)$ is the part of the unit sphere in $T_p N(\overline g)$ that is cut out by the directions to $M(g)$. If the spherical link of $p$ is a half-sphere, then $p$ is called \emph{regular} in $M(g)$. If it is a spherical lune, then $p$ is called a \emph{ridge point} in $M(g)$. Otherwise, it is called a \emph{vertex} of $M(g)$, and its spherical link is a convex spherical polygon. 
The closures of the connected components of the set of regular points in $M(g)$ are called \emph{faces} of $M(g)$ and constitute \emph{the face decomposition} of $M(g)$. We denote by $F(g)$ the set of faces of $M(g)$. The \emph{edges} of $M(g)$ are closures of the connected components of the set of ridge points of $M(g)$. 

\begin{lm}
\label{strpolyh}
Let $g$ be a strictly polyhedral metric on $M$. Then each edge is a segment between vertices and all faces are isometric to compact hyperbolic polygons.
\end{lm}

\begin{proof}
If an edge can be extended infinitely at least in one direction or is a closed geodesic, then its lift to the universal cover ends in the limit set $\Lambda(\overline g)$. Hence, the distance to $\tilde C(\overline g)$ would tend to zero along this lift, which is a contradiction as $\partial M$ is compact and the distance to the convex core of $g$ is positive. 

By the same reasoning, every face of $\tilde M(g)$ is compact. It is also simply connected. It follows that every face of $M(g)$ is isometric to a compact hyperbolic polygon.

\end{proof}


Let $g$ be a strictly polyhedral metric on $M$. Consider $\tilde M(g) \subset \H^3$ and its Gauss dual $\tilde M(g)^* \subset \dS^3$. As $g$ is strictly polyhedral, no supporting plane to $\tilde M(g)$ passes through a point of $\Lambda(\overline g)$. Indeed, this would imply the existence on $\partial \tilde M(g)$ of a curve, infinite in one direction, such that the distance to $\tilde C(\overline g)$ would tend to zero along it. Hence, each boundary component of $\partial \tilde M(g)^*$ is space-like, so the induced metric is a spherical cone-metric. Another consequence of strict polyhedrality is that $\rho_{\overline g}$ acts on $\partial \tilde M(g)^*$ freely. Indeed, a point in $\dS^3$ is fixed by $\rho_{\overline g}(\gamma)$ only if it is dual to a plane in $\H^3$ containing the axis of the isometry $\rho_{\overline g}(\gamma)$. But all such axes are contained in $\tilde C(\overline g)$. Therefore, by taking the quotient of $\partial \tilde M(g)^*$, we obtain a spherical cone-metric $d$ on $\partial M$, determined up to isotopy. We say that $d$ is the \emph{dual metric} induced by $g$ on $\partial M$.

There is another way to construct $d$. For each vertex $v$ of $M(g)$ we consider its spherical link, which is a convex spherical polygon, and its spherical polar dual is another convex spherical polygon. By gluing together these dual polygons with respect to the face decomposition of $M(g)$, we obtain $d$. It is easy to see from this description that $d$ is concave.

The main result of the work~\cite{CD} implies that if $C \subset \H^3$ is a convex polyhedral set, then any two points in $\partial C^*$ at distance less than $\pi$ can be connected by a unique minimizing intrinsic geodesic segment in $\partial C^*$. Moreover, a refined version~\cite[Theorem 4.1.1]{CD} also states that if $\gamma \subset \partial C^*$ is a closed intrinsic geodesic of length $2\pi$, then it encloses a cusp point of $C^*$, i.e., a point of tangency with $\partial_{\infty} \dS^3$. This particularly implies

\begin{lm}
Let $g$ be a strictly polyhedral metric on $M$. Then the dual metric $d$ on $\partial M$ is large for $M$.
\end{lm}



\subsection{Topology on the space of strictly polyhedral metrics}

Let $V \subset \pt M$ be a finite set, $g$ be a strictly polyhedral metric on $\partial M$ and $f: V \rightarrow F(g)$ be a bijection such that the points from the $j$-th component $S_j$ of $\pt M$ mark faces of this component. We say that the pair $(g, f)$ is \emph{a strictly polyhedral metric on $M$ with faces marked by $V$}. We denote by $\mathcal P(M, V)$ the space of strictly polyhedral metrics $g$ on $M$ with faces marked by $V$ up to isometry preserving the marking and isotopic to the identity. We will frequently abuse the notation and write $g \in \mathcal P(M, V)$ meaning that some $f$ is also chosen and with the class of $g$ it is in $\mathcal P(M, V)$. Our current goal is to endow $\mathcal P(M, V)$ with a natural topology. We will achieve this by embedding it into a larger space.

For $\overline g \in \mathcal{CH}(N)$ consider an oriented plane in $\tilde N(\overline g)=\H^3$. It gives rise to an immersed oriented plane in $N(\overline g)$. We say that an oriented plane is \emph{positive} if it is at positive distance from $\tilde C(\ol g)$ and is oriented outwards $\tilde C(\ol g)$. We say the same about the immersed image of the plane. We denote by $\mathcal{CH}(N, V)$ the space of convex cocompact hyperbolic metrics on $N$ equipped with a set of immersed positive planes marked by $V$, considered up to isometries sending marked planes to marked planes, preserving the marking and isotopic to the identity. When we say that the planes are marked by $V$, we mean that the planes in the geometric end of $N(\ol g)$ corresponding to $S_j$ are marked by points from $S_j$. The elements of $\mathcal {CH}(N, V)$ are represented by pairs $(\overline g, f)$, where $\overline g$ is a convex cocompact hyperbolic metric on $N$ and $f$ is a map from $V$ onto a set of immersed positive planes. We will write $f_v$ instead of $f(v)$. We are going now to endow the space $\mc{CH}(N, V)$ with a natural topology.

First, we have to define similarly a space $\mc{CH}_h(N, V)$ with isotopy replaced by homotopy. As with $\mc{CH}_h(N)$, it is easier to endow $\mc{CH}_h(N, V)$ with a topology. We denote by $(\dS^3)^{V}$ the space of maps $\tilde f: V \ra \dS^3$. We can interpret $(\dS^3)^{V}$ as the configuration space of oriented planes in $\H^3$ marked by $V$. We consider the direct product
$$\mc R(\pi_1(M), G)\times (\dS^3)^{V}$$
with a left action of $G$ on it, by conjugation on the first factor and by isometries on the second factor. There is also a left action of $\pi_1(M)$ given by $\gamma(\rho, \tilde f)=(\rho, \rho(\gamma)\tilde f)$, where $(\rho, \tilde f) \in \mc R(\pi_1(M), G)\times (\dS^3)^{V}$ and $\gamma \in \pi_1(M)$. These two actions commute, and the standard approach via developing maps allows us to inject $\mc{CH}_h(N, V)$ in the quotient. Since the action of $G$ on the holonomies of convex cocompact metrics is free and proper (see Section~\ref{cocomp}), and for every such holonomy $\rho$, its action on the space of positive planes is free and properly discontinuous, this endows $\mc{CH}_h(N, V)$ with the topology of a smooth manifold of dimension $3(n-k)$ where $k$ is the Euler characteristic of $\partial M$.

For $\ol g \in \mc{CH}_h(N, V)$ the fiber of the forgetful map $\mc{CH}_h(N, V) \ra \mc{CH}_h(N)$ can be naturally identified with the space of positive planes in $N(\ol g)$ marked by $V$. Recall that the action of ${\rm MCG}_h(M)$ on $\mc{CH}(N)$ is free, properly discontinuous and determines a covering map $\mc{CH}(N) \ra \mc{CH}_h(N)$. Consider the commutative diagram
\begin{center}
\begin{tikzcd}
\mathcal{CH}(N, V) \arrow[r] \arrow[d]
& \mathcal{CH}_h(N, V) \arrow[d] \\
\mathcal{CH}(N) \arrow[r]
& \mathcal{CH}_h(N)
\end{tikzcd}
\end{center}
Using this diagram, we pull-back the topology on $\mathcal{CH}_h(N, V)$ to a topology on $\mc{CH}(N, V)$ making it also a smooth manifold of dimension $3(n-k)$. 
In Section~\ref{bundlesec} we will show that the vertical forgetful maps are fiber bundles.


For $g \in \mathcal P(M, V)$ every face gives rise to an immersed plane, which we consider oriented outwards $M(g)$. As shown in Section~\ref{polyhsec}, the lift of any such plane does not contain points of $\Lambda(\ol g)$ at infinity, hence, these planes are at positive distance from $C(\ol g)$. This allows us to consider $\mathcal P(M, V)$ as a subset of $\mathcal{CH}(N, V)$. Clearly, it is open. We endow $\mathcal P(M, V)$ with the induced topology. 

\begin{lm}
\label{nonempty}
For every admissible $M$ if $V$ has sufficiently many points in every component of $\pt M$, then the space $\mc P(M, V)$ is non-empty.
\end{lm}

\begin{proof}
We show even more, that for every $\ol g \in \mc{CH}(N)$ one can choose sufficiently large $V$ so that there exists $g \in \mathcal P(M, V)$ such that $\ol g$ is the underlying metric on $N$. Indeed, for a component $S_j$ of $\pt M$ let $\nu_j: \pi_1(S_j) \ra \pi_1(M)$ be the push-forward homomorphism induced by the inclusion. Consider $\tilde N(\ol g) \cong \H^3$, there exists a component $\Omega_j$ of $\pt_{\infty} \H^3 \backslash \Lambda(\ol g)$ invariant under $\rho_{\ol g}|_{\im(\nu_j)}$. There exists a compact fundamental domain in $\Omega_j$ for the latter action. As $\Omega_j$ is open in $\pt_\infty \H^3$, there exists a finite cover of the chosen fundamental domain by disks in $\pt_\infty \H^3$ contained in $\Omega_j$. Every such disk determines a plane in $\H^3$ that is at positive distance from $\tilde C(\ol g)$. We do this for every component of $\pt M$ and then consider the union of $\rho_{\ol g}$-orbits of all chosen planes. Take the intersection of all the half-spaces containing $\tilde C(\ol g)$ bounded by these planes. It is evident that the $\rho_{\ol g}$-quotient of this set is homeomorphic to $M$ and is a strictly polyhedral manifold.
\end{proof}

For every pair $(\ol g, f) \in \mc{CH}(N, V)$ we consider the intersection $K$ of all the closed negative immersed half-spaces in $N(\ol g)$ bounded by the immersed marked planes. We denote by $\ol{\mc P}(M, V)$ the subset of $\mc{CH}(N, V)$ such that (1) $K$ is homeomorphic to $M$; (2) $f$ is injective; (3) each marked plane is supporting to $K$. We remark that here we use a non-intuitive notation as $\ol{\mc P}(M, V)$ is not a topological closure of $\mc P(M, V)$: the latter also contains configurations with non-injective markings and with cusps. Our motivation for this notation is that an element of $\ol P(M, V)$ determines a dual metric in $\ol{\mc D}_c(\pt M, V)$. We denote by $g$ the metric on $M$ induced by a homeomorphism between $M$ and $K$, where a homeomorphism is chosen so that when it is composed with our fixed homeomorphism $N \cong \inter(M)$, we get the identity on the fundamental groups. The metric $g$ is determined up to isotopy. We will represent the elements of $\ol{\mc P}(M, V)$ by pairs $(g, f)$. For every such pair there is a subset $V(g) \subseteq V$ where $v \in V(g)$ if and only if the intersection of the plane $f_v$ with $\pt M(g)$ has non-empty relative interior, i.e., corresponds to a face. A pair $(g, f)$ also determines a class in $\mc P(M, V(g))$.

For $g \in \mathcal P(M, V)$ the dual metric $d$ is defined on $\pt M$ up to isotopy, so it determines an element of $\mathcal D_M(\partial M, V)$. For a pair $(g, f)$ representing an element of $\ol{\mc P}(M, V)$ the dual points to the lifts of the planes $f_v$ belong to the boundary of the dual set $\tilde M(g)^*$. Hence, in this case we may consider the class of the dual metric in $\ol{\mc D}_M(\pt M, V)$. We denote the obtained map by $$\mathcal I_V: \ol{\mathcal P}(M, V) \rightarrow \ol{\mathcal D}_M(\partial M, V).$$
It is not hard to see that it is continuous. We now show that it is $C^1$ over $\mc P(M, V)$.

First, if $g \in \mathcal P(M, V)$ and $e$ is an edge of $M(g)$, then $e$ remains to be an edge in some neighborhood of $g$ in $\mathcal P(M, V)$. From our definition of the topology on $\mathcal P(M, V)$, when we vary $g$ smoothly, the planes containing faces of $\tilde M(g)$ vary smoothly. This implies

\begin{lm}
\label{diffweak}
The dihedral angle of $e$ is a smooth function in a neighborhood of $g$ in $\mathcal P(M, V)$.
\end{lm}

Now we prove

\begin{lm}
\label{differ}
The map $\mathcal I_V$ is $C^1$ over $\mc P(M, V)$.
\end{lm}

\begin{proof}
Let $g \in \mathcal P(M, V)$ and $d$ be the respective dual metric. The boundary of $\tilde M(g)^*$ is decomposed into faces, which are compact convex spherical polygons. It projects to a decomposition of $(\partial M, V)$ into faces. Let $\mathcal T$ be any triangulation of $(\partial M, V)$ refining this decomposition. The metric $d$ is $\mc T$-triangulable. We need to show that the edge-lengths of $\mathcal T$ are $C^1$-functions of $g$. Due to Lemma~\ref{diffweak}, it remains only to check what happens when there are non-triangular faces in the face decomposition of $\tilde M(g)^*$, or, equivalently, when there are non-simple vertices of $\tilde M(g)$. 


Let $Q \subset \dS^3$ be a convex geodesic space-like polygon with more than three vertices representing a face of $\tilde M(g)$, $V_Q$ be the set of its vertices, $\tilde e$ be its interior edge with endpoints $v_1$ and $v_2$ representing an edge $e$ of $\mc T$. We denote by $s$ the de Sitter distance between $v_1$ and $v_2$.
When we slightly vary the positions of $V_Q$ and take the convex hull, we may obtain a 3-dimensional polytope, whose boundary is naturally divided into two parts, both of which can be considered as variations of $Q$. Let us restrict the attention to one of the parts, we call it \emph{upper}. If the deformation is small enough, this polytope has only space-like faces, and the total angles at the vertices of the upper part in the intrinsic metric stay less than $\pi$. Let $X_Q \subset (\dS^3)^{V_Q}$ be a small enough neighborhood of the initial position of $V_Q$, for which this holds. Then we can identify a deformation of $\tilde e$ in the upper part, which is a broken line if $\tilde e$ is not an edge of the obtained polytope. Consider the length $\tilde l$ of $\tilde e$ as a function over $X_Q$. There exists a small neighborhood $Y$ of $g$ in $\mc P(M, V)$, over which the length $l$ of $e$ can be defined, and a smooth map $q: Y \ra X_Q$ such that $l=\tilde l \circ q$. Thus, it is enough to show that $\tilde l$ is a $C^1$-function.

Now let $\mathcal T_1, \ldots, \mathcal T_r$ be all possible triangulations of $Q$. Fix some $\mathcal T_i$ and consider deformations of $Q$ triangulated with $\mathcal T_i$ under small variations of $V_Q$. Here the obtained polyhedral surface may become non-convex. We assume that $X_Q$ is small enough that for all $\mathcal T_i$ all faces of the deformed surface are space-like and all angles of the vertices in the intrinsic metric are less than $\pi$. Let $l_i$ be the length of $\tilde e$ under such deformations. Clearly, it is a smooth function in $X_Q$ as all the edge-lengths of $\mathcal T_i$ are smooth functions. For each $v \in V_Q$ we choose coordinates such that two of them, called \emph{horizontal}, are the coordinates of the projection of $v$ to the plane containing $Q$ in the initial moment, and the third, called \emph{vertical}, is the oriented length of a perpendicular from $v$ to this plane. Note that since the plane is space-like, the perpendicular is time-like, but here we consider its length as a real number (with the sign depending on the position of $v$ with respect to the plane). 

We claim that all vertical partial derivatives of $l_i$ at the initial moment are zero. Indeed, let $ABC$ be a de Sitter triangle with $BC$ time-like, $AB$ and $AC$ space-like and $AB$ orthogonal to $BC$. Denote by $a$, $b$, $c$ the respective side-lengths (again, we treat the length of $BC$ as a real number). Then we have
$$\cos(b)=\cos(c)\cosh(a).$$
Thus, for fixed $c \neq 0$, $\partial b/\partial a=0$ at $a=0$. Thus, the vertical partial derivatives of all edge-lengths of $\mathcal T_i$ are zero. In turn, this shows that the vertical partial derivatives of $l_i$ are zero too. Now we remark that all horizontal partial derivatives of $l_i$ in the initial moment coincide with the horizontal partial derivatives of $s$, the de Sitter distance between the endpoints of $\tilde e$.

We return to the convex hull of $V_Q$. The domain $X_Q$ admits a piecewise-smooth decomposition into cells $E_1, \ldots, E_r$ corresponding to the face triangulations of the upper part of $\conv(V_Q)$. The function $\tilde l$ coincides with $l_i$ over $E_i$. The argument above shows that in the initial moment the differential of each $l_i$ coincides with the differential of $s$. Thus, $\tilde l$ is differentiable in the initial moment. Moreover, it follows that at any point of $X_Q$ belonging to several $E_i$, the differentials of the respective $l_i$ coincide. This implies that $\tilde l$ is differentiable over $X_Q$. Since all partial derivatives of $l_i$ are continuous over $E_i$ and coincide with the partial derivatives of $E_j$ at the common points of $E_i$ and $E_j$, we see that $\tilde l$ is $C^1$. This finishes the proof.
\end{proof}

We can refine now Theorem~\ref{mt} to a more appropriate form:

\begin{thm}
\label{mtr}
For every admissible $M$ and every $V \subset \pt M$ such that $\mc D_M(\pt M, V)$ is non-empty, the space $\mc P(M, V)$ is non-empty and the restriction of the map $\mathcal I_V$ to $\mc P(M, V)$ is a $C^1$-diffeomorphism onto $\mc D_M(\pt M, V)$.
\end{thm}

In the next section we will show that $d\mathcal I_V$ is non-degenerate over $\mc P(M, V)$. Next, in Section~\ref{psec} we will show that $\mathcal I_V$ is proper. We finish the proof of Theorem~\ref{mtr} in Section~\ref{fsec}.

\section{Infinitesimal rigidity}
\label{irsec}

The goal of this section is to prove 

\begin{lm}
\label{infrig}
If $x \in T_g\mathcal P(M, V)$ is such that $d\mathcal I_V(x)=0$, then $x=0$.
\end{lm}

Take two copies of $M$ and identify the corresponding boundary points reversing the orientation. We obtain a closed 3-manifold $M^D$ called \emph{the double of $M$}. A strictly polyhedral hyperbolic metric $g$ on $M$ produces a metric $g^D$ on $M^D$, which is a hyperbolic cone-metric with the singular locus coming from the edges of $M(g)$ and with all cone-angles smaller than $2\pi$. We refer to~\cite{HK}, \cite{Wei} or \cite{MM3} for introductions to hyperbolic cone-3-manifolds. To uniformize the notation, we will write $M^D(g^D)$ for the hyperbolic cone-3-manifold $(M^D, g^D)$. We denote by $\Sigma(g^D)$ the singular locus of $g^D$ and denote by $L(g^D):=M^D(g^D) \backslash \Sigma(g^D)$, the smooth part of $M^D(g^D)$ equipped with the induced metric. 

We claim that if $g_0$, $g_1$ belong to the same connected component of $\mathcal P(M, V)$, then $L(g^D_0)$, $L(g^D_1)$ are isotopic as subsets of $M^D$.
Indeed, for every $g \in \mc P(M, V)$ that has a vertex of degree $>3$ and for every four marked planes in $\tilde N(\ol g)$ intersecting in the same point, the subset of $\mc P(M, V)$, consisting of $g' \in \mc P(M, V)$ such that these planes continue to intersect in the same point in $\tilde N(\ol g')$, is a smooth submanifold of codimension 1 in $\mc P(M, V)$ around $g$.
(Furthermore, this submanifold is analytic for a natural analytic structure on $\mc P(M, V)$.) Locally in $\mc P(M, V)$ there are finitely many such submanifolds. Therefore, for any two $g_0, g_1$ belonging to the same component of $\mc P(M, V)$, we can connect them by a smooth path $g_t$ transversal to all such submanifolds. Along the path some vertices of $\pt M(g_t)$ can merge or split. In both cases one can see that these changes do not affect the isotopy type of $L(g^D_t)$ in $M^D$.

We need now some preliminaries on the deformation spaces of hyperbolic cone-3-manifolds similar to our treatment of convex cocompact hyperbolic metrics in Section~\ref{cocomp}. Fix a connected component $\mc P_0$ of $\mathcal P(M, V)$. Let $L$ be a smooth open manifold diffeomorphic to $L(g^D)$ for $g \in \mc P_0$. Note that $L$ is diffeomorphic to the interior of a compact manifold with boundary, obtained by deleting from $M^D(g^D)$ a small regular neighborhood of the singular set. We denote by $\mathcal H(L)$ the space of hyperbolic metrics on $L$ up to equivalence, where the equivalence is generated by isotopies and thickenings. Here the metrics are not necessarily complete. If a hyperbolic metric $g_1$ on $L$ extends to a hyperbolic metric $g_2$ on $L \cup (\partial L \times [0, \e))$ for $\e>0$, then $g_2$, considered as a metric on $L$, is called \emph{a thickening} of $g_1$. We refer to~\cite[Section 1.6]{CEG} or to~\cite[ Section 5.3]{CHK} for more details on this equivalence. We endow the space $\mc H(L)$ with the topology of convergence of the developing maps on compact subsets of $\tilde L$. We denote by $\mathcal C(M^D) \subset \mathcal H(L)$ the subspace of equivalence classes that can be represented by metrics, whose metric completion is a hyperbolic cone-manifold homeomorphic to $M^D$. 

Recall that the group $G$ of the orientation-preserving isometries of $\H^3$ is isomorphic to ${\rm PSL}(2, \C)$. It is convenient now to consider its universal cover $\tilde G\cong {\rm SL}(2, \C)$. As usual, a holonomy map $\pi_1(L) \rightarrow G$, defined up to conjugation, is assigned to every hyperbolic metric from $\mathcal H(L)$. It can be lifted to a holonomy map $\pi_1(L) \rightarrow \tilde G$ by a result of Culler~\cite{Cul}. Again, we consider the space of representations $\mathcal R(\pi_1(L), \tilde G)$ and its quotient $\mathcal X(\pi_1(L), \tilde G)$ by conjugation of $\tilde G$. The characters of holonomies determine a map $\mathcal H(L) \rightarrow \mathcal X(\pi_1(L), \tilde G)$. A version of the Ehresmann-Thurston theorem~\cite[Theorem 1.7.1]{CEG} (see also~\cite[Section 5.3]{CHK}) shows that this map is a local homeomorphism. The works of Montcouquiol and Weiss~\cite{Mon, Wei2} show that around the characters of cone-manifolds, $\mathcal X(\pi_1(L), \tilde G)$ locally has the topology of a complex manifold of complex dimension $3(n-k)$ (recall that $k$ is the Euler characteristic of $\partial M$). We endow $\mc H(L)$ with the pull-back of this smooth structure.

The map $g \mapsto g^D$ determines a map $D: \mc P_0 \rightarrow \mathcal C(M^D) \subset \mathcal H(L)$. We can show

\begin{lm}
\label{double}
The map $D$ is a smooth immersion.
\end{lm}

\begin{proof}
Let $g \in \mathcal P_0$, fix a natural embedding $\iota:\inter(M) \hookrightarrow L$, and let $p \in \tilde L$ be a point projecting to $\iota(\inter(M))$, $(e_1, e_2, e_3 ) \in T_{p} \tilde L$ be a basis in its tangent space. Consider $\gamma \in \pi_1(L)$. From the description of topology of $\mathcal P(M, V)$ it follows that when $g$ varies smoothly, $\gamma(p)$ and $\gamma_*(e_i)$ vary smoothly in $\tilde L$ for all $i=1,2,3$. Fix a developing map ${\rm dev}: \tilde L(g^D) \rightarrow \H^3$. Consider a variation of $g$ and a respective variation of ${\rm dev}$, we can always do it so that  the metric on $T_p \tilde L$ does not change, and ${\rm dev}$ is constant on $p$ and $(e_1, e_2, e_3)$. Then ${\rm dev}(\gamma (p))$ and all ${\rm dev}_*(\gamma_* (e_i))$ vary smoothly. Every isometry of $\H^3$ is completely determined by its action on a point of $\H^3$ and on a basis in the tangent space of this point. Hence, $\rho(\gamma)$ varies smoothly, where $\rho$ is the respective holonomy representation of $\pi_1(L)$ in $G$. This shows that $D$ is a smooth map.

Now suppose that there is $x \in T_g \mathcal P(M, V)$ such that $dD(x)=0$. Denote by $A$ the closure of $\iota(\inter(M))$ in $L$. The inclusion $A \hookrightarrow L$ is a retract. Indeed, the complement $L \backslash A$ is naturally homeomorphic to $\inter(M)$. We send every point of $L \backslash A$ to the corresponding point of $\iota(\inter(M))\subset A$, this produces a retraction map. Hence, it induces an injection of fundamental groups $\pi_1(M)\cong\pi_1(A) \hookrightarrow \pi_1(L)$. Thus, $x$ induces zero variation at the first order on the holonomy representation $\rho_{\overline g}$ of $\pi_1 (M)$, and, consequently, induces zero variation at the first order on $\overline g$. 

It follows that $x$ is vertical with respect to the forgetful submersion $\mathcal {CH}(N, V) \rightarrow \mathcal {CH}(N)$. Thus, $x$ can be represented by a Killing vector field $X_v$ on each immersed marked plane $f_v$ in $N(\overline g)$. Let $\tilde f_v \subset \H^3$ be a lift of $f_v$ to $\tilde N(\overline g)\cong\H^3$. The vector field $X_v$ lifts to a vector field $\tilde X_{\tilde f_v}$, which is the restriction of a global Killing field of $\H^3$ to the plane $\tilde f_v$. As $dD(x)=0$, $x$ induces zero first-order variation on all lengths and dihedral angles of $M(g)$, including all edges of zero length at vertices of valence greater than three. First, it implies that for every vertex $u$ of $\tilde M(g)$ the restriction to $u$ of each vector field $\tilde X_{\tilde f_v}$ coincide. Indeed, in the dual picture in $\dS^3$ we have a face dual to $u$ and a tangent vector at every vertex such that these vectors preserve this face infinitesimally. This means that they are restrictions of a global Killing field of $\dS^3$  to this face. In the non-dual picture this exactly means that the vector fields coincide on vertices, and thereby, on the edges too. Since all the dihedral angles are infinitesimally preserved, the obtained vector field on every component of $\pt\tilde M(g)$ is the restriction of a global Killing field on $\H^3$. On the other hand, it should be invariant with respect to the restriction of $\rho_{\ol g}$ to the subgroup preserving the given component. The only such option is that this Killing field is zero. Hence, all $X_v$ are zero, which implies that $x$ is zero.
\end{proof}

Now we quickly recall the basics of the infinitesimal theory of the character space $\mathcal X(\pi_1(L), \tilde G)$ via the bundle of infinitesimal isometries on $L$. This goes back to the works of Weil~\cite{Weil} and Matsushima--Murakami~\cite{MM} notably expounded in the book~\cite{Rag} of Raghunathan. The applications of this theory to cone-manifolds are due to Hodgson--Kerckhoff, Montcouquiol and Weiss~\cite{HK, Wei, Wei2, MW}. We describe the theory quite briefly, and mostly refer to the above-mentioned fundamental works.

Until the end of this section we fix some $g \in \mc P(M, V)$ and we will write just $M$, $M^D$, $L$, $\Sigma$ instead of $M(g)$, $M^D(g^D)$ and so on. Fix a developing map from $\tilde L$ to $\H^3$, let $\rho \in \mathcal R(\pi_1(L), \tilde G)$ be a lift of its holonomy and $\chi \in \mathcal X(\pi_1(L), \tilde G)$ be its character. As mentioned above, $\mathcal X(\pi_1(L), \tilde G)$ is a complex manifold around $\chi$. The work of Weil~\cite{Weil} establishes the identification
$$T_{\chi}\mathcal X(\pi_1(L), \tilde G)\cong H^1(\pi_1(L), \Ad\circ \rho),$$
where $H^1(\pi_1(L), \Ad\circ \rho)$ is the first cohomology group of $\pi_1(L)$ with respect to the representation $\Ad \circ \rho$ on the Lie algebra $\mathfrak g=\mf{sl}(2, \C)$ of $\tilde G$, and $\Ad$ is the adjoint representation of $\tilde G$ on $\mathfrak g$. See the details in~\cite[Chapter VI]{Rag},~\cite[Section 4]{HK},~\cite[Section 6.2]{Wei}.

Let $\tilde E \rightarrow \tilde L$ be the trivial bundle, where $\tilde E$ is identified with $\tilde L \times \mathfrak g$ and $\tilde L$ is endowed with the metric $g^D$. The trivialization furnishes the bundle with a flat connection. It is also endowed with a natural, yet non-parallel metric, see, e.g.,~\cite[Section 1]{HK}. The group $\pi_1(M)$ acts on $\tilde E$ via the representation $\Ad \circ \rho$. By taking the quotient, we obtain a flat bundle $E \rightarrow L$ of the infinitesimal isometries of $L$. The flat connection induces an exterior derivative $d^E$ on the complex $\Omega^\bullet(L, E)$ of $E$-valued differential forms on $L$. Consider the respective cohomology complex $H^\bullet(L, E)$. An analogue of the de Rham Theorem in this setting establishes that
$$H^1(\pi_1(L), \Ad\circ \rho) \cong H^1(L, E).$$
The details can be found in~\cite[Chapter VII]{Rag},~\cite[Section 1]{HK},~\cite[Section 6]{Wei}. We denote by $H^1_{L^2}(L, E)$ the first $L^2$-cohomologies, i.e., the cohomologies of the  subcomplex of $L^2$-bounded differential forms with $L^2$-bounded differential.

For $\gamma \in \pi_1(L)$ denote by $\tr_\gamma$ the trace of $\rho(\gamma)$.  As trace is invariant with respect to conjugation, it is defined for $\chi$ and is defined for free homotopy classes of closed curves. Also the trace of an element of $\tilde G$ is equal to the trace of inverse, thereby, $\tr_\gamma$ does not depend on the orientation of $\gamma$. Recall that for $\gamma \in \pi_1(L)$, the element $\rho(\gamma)$ is elliptic if and only if its trace $\tr_{\gamma} \in (-2, 2)$. Moreover, in this case $\tr_{\gamma}=2\cos(2\theta)$, where $2\theta$ is the rotation angle of $\rho(\gamma)$. For an edge $e$ of $\Sigma$ we denote by $\gamma_e$ the class of the meridian around $e$, i.e., of the boundary of a small disk orthogonal to $e$. 

Fix some sufficiently small $r_0$ such that the closed $r_0$-neighborhoods of vertices do not intersect, and the closed $r_0$-neighborhoods of edges intersect only in the neighborhoods of the adjacent vertices. We will denote the smooth part of the open $r$-neighborhood of a set $X \subset M^D$ by $B_r(X)$, and denote by $\bar B_r(X)$ its closure in $L$. For a vertex $u$ of $\Sigma$ we additionally denote $\Xi_u:=\pt B_{r_0}(u)$.

Let $\Xi$ be a sphere with finitely many punctures, and $\rho: \pi_1(\Xi) \rightarrow \tilde G$ be an irreducible representation, i.e., it does not leave invariant any plane in $\H^3$. 
A pair of pants decomposition $\mc P$ of $\Xi$ is called \emph{$\rho$-admissible} if the restriction of $\rho$ to each pair of pants is irreducible. We will make use of the following result

\begin{lm}[\cite{MW}, Lemma 2.5]
\label{linind}
Let $\Xi$ be a punctured sphere, $\mc Q$ be a set consisting of one peripheral curve for each puncture, $\rho: \pi_1(\Xi) \rightarrow \tilde G$ be an irreducible representation and $\mathcal P$ be a $\rho$-admissible pair of pants decomposition. Then $\{d\tr_\gamma\}_{\gamma \in \mc P \cup \mc Q}$ are $\C$-linearly independent in $T^*_\chi \mc X(\pi_1(\Xi), \tilde G)$ at $\chi=[\rho]$.
\end{lm}

For $\alpha_1, \alpha_2, \alpha_3 \in (0, 2\pi)$ we denote by $\Xi(\alpha_1, \alpha_2, \alpha_3)$ the smooth part of a (unique) spherical cone-surface with exactly three cone-angles of given values, if it exists. (It is easy to see that it exists if and only if the sum of the two largest angles is smaller than the smallest one plus $\pi$.) Now let $\Xi$ be the smooth part of an oriented convex spherical cone-surface with at least 3 vertices. Let $\mc P$ be a pair of pants decomposition of $\Xi$. Assume that the holonomy angle of every curve in $\mc P$ is $<2\pi$ and for every pair of pants $\Xi'$ of $\mc P$ there exists the respective surface $\Xi(\alpha_1, \alpha_2, \alpha_3)$, where the angles are either the holonomy angles of the boundary curves of $\Xi'$ or the cone-angles of vertices. A developing map allows to immerse the surface $\Xi'$ into $\Xi(\alpha_1, \alpha_2, \alpha_3)$. Assume additionally that this immersion is an embedding for every pair of pants $\Xi'$. If all our assumptions hold, we say that $\mc P$ is \emph{suitable}. Now we can formulate our main technical lemma

\begin{lm}
\label{cohomol}
For every vertex $u$ of $\Sigma$ let $\mc P_u$ be a suitable $\rho$-admissible pair of pants decomposition of $\Xi_u$. If $y \in H^1(L, E)$ satisfies $d\tr_{\gamma_e}(y)=0$ for every edge $e$ of $\Sigma$ and $d\tr_\gamma(y)=0$ for every $\gamma$ in every $\mc P_u$, then $y=0$.
\end{lm}

Note that the induced metric on $\Xi_u$ have positive constant Gaussian curvature, which may be different from one, but we check the suitability by scaling the metric to make it spherical.

\begin{rmk}
We think that the condition that $\mc P_u$ are suitable is redundant. The suitability, however, allows to use directly subtle analytic bounds from~\cite{Wei2}, namely, Lemma~\ref{anest} below. Moreover, in our situation it is straightforward that the pair of pants decompositions under consideration are suitable. Without suitability we may try to develop a small neighbourhood of each curve $\gamma$ to a respective spherical suspension and perform there a similar analysis as it is done below. This should allow us to obtain slightly worse (due to spectral properties of spherical suspensions, see~\cite[Lemma 3.18]{Wei2}), but still sufficient bounds. We choose not to pursue this path. 
\end{rmk}

We now deduce Lemma~\ref{infrig} from Lemma~\ref{cohomol}, and then prove Lemma~\ref{cohomol}.

\begin{proof}[Proof of Lemma~\ref{infrig}.]
Denote $d D(x)$ by $y$, considered as an element of $H^1(L, E)$. As $y$ is tangent to the $D$-image of $\mathcal P(M, V)$, in particular, it is tangent to $\mathcal C(M^D)$. Then for all meridians $\gamma_e$ we get $\Im(d\tr_{\gamma_e})(y)=0$. Next, as $d\mathcal I_V(x)=0$, then for the dihedral angle $\theta_e$ of each edge we have $d\theta_e(x)=0$. Thus, $\Re(d\tr_{\gamma_e})(y)=0$, so $d\tr_{\gamma_e}(y)=0$.

Consider a vertex $u$ of $\Sigma$. Look at $u$ also as at a vertex of $M$. Its spherical link in $M$ is a convex spherical polygon $Z$, and its spherical link in $M^D$ is a convex spherical cone-surface $Z^D$, which is the double of $Z$. Denote also by $Z^*$ the dual spherical polygon to $Z$. The tangent vector $x$ induces the decomposition of $Z^*$ into convex polygons (without adding new vertices), which we extend to a triangulation. For each interior edge of this triangulation we consider an arbitrary smooth curve in $Z$ connecting the respective edges of $Z$ and orthogonal to them, such that these curves are disjoint for different interior edges of the triangulation. This produces a decomposition of $Z$ into curvilinear polygons. The double of this decomposition is a pair of pants decomposition of the smooth part of $Z^D$, which gives a pair of pants decomposition $\mc P_u$ of $\Xi_u$. Note that it is suitable: this is easy to see since every curvilinear polygon $C$ in $Z$ belongs to the convex spherical triangle formed by the extensions of those edges of $C$ that come from the edges of $Z$ (note that there are always three such edges). Since $d\mc I_V(x)=0$, by construction, $d\tr_\gamma(y)=0$ for every $\gamma \in \mc P_u$.

It is also easy to see that $\mathcal P_u$ is admissible for the restriction of $\rho$ to $\pi_1(\Xi_u)$. Indeed, its restriction to any pair of pants is conjugate to the holonomy of a spherical cone-metric obtained by gluing two copies of a convex spherical triangle. Since the cone-angles are smaller than $2\pi$ and the side-lengths of the triangle are smaller than $\pi$, the holonomy can not leave invariant a plane in $\H^3$.

Now Lemma~\ref{cohomol} shows that $y=0$. Then Lemma~\ref{double} implies that $x=0$.
\end{proof}

Our proof of Lemma~\ref{cohomol} will employ the Bochner Technique and will follow the proof of~\cite[Theorem 3.15]{Wei2}, though there will be some differences.
We need to make more preparations. Pick some $\Xi_u=\Xi$ with the restriction of the bundle $E$ to it and a $\rho$-admissible pair of pants decomposition $\mc P_u=\mc P$. Let $\mc Q_u=\mc Q$ be a set consisting of one peripheral curve in $\Xi$ for each cone-point. From~\cite[Lemma 3.10]{Wei2}
\begin{equation}
\label{c1}
H^1_{L^2}(\Xi, E) \subset \bigcap_{\gamma \in \mc Q}\ker(d\tr_\gamma).
\end{equation}
Here and in what currently follows the kernels are taken in $H^1(\Xi, E)$. Due to~\cite[Lemma 3.9]{Wei2}, we have 
\begin{equation}
\label{c2}
\dim_\C(H^1(\Xi,E))=3(\deg(u)-2),~~~~~\dim_\C(H^1_{L^2}(\Xi, E))=2(\deg(u)-3).
\end{equation}
By Lemma~\ref{linind}, $\{d\tr_\gamma\}_{\gamma \in \mc Q}$ are $\C$-linearly independent in $H^1(\Xi, E)^*$, where $*$ means the dual vector space. This,~(\ref{c1}) and~(\ref{c2}) show that
$$H^1_{L^2}(\Xi, E) =\bigcap_{\gamma \in \mc Q}\ker(d\tr_\gamma).$$
Then Lemma~\ref{linind} also implies that $\{d\tr_\gamma\}_{\gamma \in \mc P}$ are $\C$-linearly independent in $H^1_{L^2}(\Xi, E)^*$. We have 
\begin{equation}
\label{c3}
\dim_\C\left(H^1_{L^2}(\Xi, E) \cap \left(\bigcap_{\gamma \in \mc P}\ker(d\tr_\gamma)\right)\right)=\deg(u)-3.
\end{equation}
Denote this space by $Y_u=Y$.

The trace form $h$ on $\mf g$ is $\Ad$-invariant. Thereby, it induces a parallel non-degenerate $\C$-valued form $h^E$ on $E$. For $\omega_1, \omega_2 \in \Omega^1_{L^2}(\Xi, E)$ consider the expression 
\begin{equation}
\label{c4}
H(\omega_1, \omega_2):=\int_\Xi h^E(\omega_1\wedge \omega_2).
\end{equation}
From the work of Cheeger~\cite{Che} the $L^2$-Stokes Theorem holds for $\Xi$ (it is important here that $\Xi$ is 2-dimensional, so the links of cone-points are just circles). Thereby, $H$ descends to a skew-symmetric form on $H^1_{L^2}(\Xi, E)$, which we continue to denote by $H$. For $A \in \tilde G$ let $F(A) \in \mf g$ be an element such that for every $X \in\mf g$ we have $h(F(A), X)=d\tr_A(X)$.

Let $\gamma$ be a smooth simple closed curve on $\Xi$. Orient it arbitrarily and pick a smooth collar $\bar C_\gamma \subset \Xi$ to the left of $\gamma$ parametrized as $\gamma \times [0,1]$. Let $\phi: [0,1] \ra [0,1]$ be a smooth function equal to 0 in a neighborhood of zero and to 1 in a neighborhood of one. Let $\phi_\gamma: \bar C_\gamma \ra \R_+$ be a smooth function that is equal to $\phi$ on every $[0,1]$-fiber of the parametrization of $\bar C_\gamma$. For $p \in \gamma$ we pick an identification of the fiber of $E$ over $p$ with $\mf g$ and choose there the element $F(\rho(\gamma))$. Since $F(\rho(\gamma))$ is $\Ad(\rho(\gamma))$-invariant, it defines a parallel section $\sigma_\gamma$ over $\bar C_\gamma$. Define $\omega_\gamma:=d\phi_\gamma \otimes \sigma_\gamma \in \Omega^1_{L^2}(\Xi, E)$, which is a smooth compactly supported 1-form. Since $\sigma_\gamma$ is a parallel section, $\omega_\gamma$ is closed. Denote its class by $y_\gamma \in H^1_{L^2}(\Xi, E)$. A simple computation using the Stokes Theorem shows that 
\begin{equation}
\label{c5}
H(y_\gamma, y)=d\tr_\gamma(y)
\end{equation}
for every $y \in H^1_{L^2}(\Xi, E)$, see~\cite[p. 378]{MW}. 

We pick disjoint collars $\bar C_\gamma$ for all $\gamma \in \mc P$ and produce the 1-forms $\omega_\gamma$ by the above procedure. Since $\{d\tr_\gamma\}_{\gamma \in \mc P}$ are $\C$-linearly independent in $H^1_{L^2}(\Xi, E)^*$, (\ref{c5}) implies that the classes $y_\gamma$ are $\C$-linearly independent in $H^1_{L^2}(\Xi, E)$. Further, the disjointness of the supports of $\omega_\gamma$, the skew-symmetry of $H$, (\ref{c4}) and (\ref{c5}) mean that $y_\gamma \in \bigcap_{\gamma' \in \mc P}\ker(d\tr_{\gamma'})$ for every $\gamma \in \mc P$. 
This and (\ref{c3}) result in

\begin{lm}
\label{canform}
The classes $y_\gamma$, $\gamma \in \mc P_u$, form a basis for $Y_u$.
\end{lm} 

Now we return to the space $L$. We have

\begin{lm}[\cite{Wei2}, Corollary 4.17]
\label{L_2}
$$H^1_{L^2}(L, E)=\bigcap_e \ker(d\tr_{\gamma_e}),$$
where the intersection is over all edges of $\Sigma$ and the kernels are taken in $H^1(L, E)$.
\end{lm}

Furthermore,~\cite[Corollary 2.3 and Lemma 3.14]{Wei2} imply the $L^2$-Hodge Theorem 

\begin{lm}
\label{hodge}
Every class in $H^\bullet_{L^2}(L, E)$ has a unique representative $\omega \in \Omega^\bullet_{L^2}(L, E)$ satisfying $d^E \omega=\delta^E \omega=0$.
\end{lm}

We need to recall some information about $d^E$ and $\delta^E$, the differential and the codifferential on $\Omega^\bullet(L, E)$. Let $(e_1, e_2, e_3)$ be an orthonormal frame on $B_{r_0}(u)$ for a vertex $u$ of $\Sigma$, where $e_3=\pt/\pt r$ and $r$ is the radial coordinate. Consider the operators
$$D=\sum_{i=1}^3 \e(e^i)\nabla_{e_i},~~~~~T=\sum_{i=1}^3 \e(e^i)\ad(E_i),$$
$$D^t=-\sum_{i=1}^3 \iota(e_i)\nabla_{e_i},~~~~~T^t=\sum_{i=1}^3 \iota(e_i)\ad(E_i),$$
where $\e(.)$ and $\iota(.)$ are the extrinsic and intrinsic products on forms, and $E_i$ is the infinitesimal translation in the direction $e_i$.
Then $d^E=D+T$ and $\delta^E=D^t+T^t$. Note that $T$ and $T^t$ are bounded 0th-order operators, see details in~\cite{HK, MM, Wei, Wei2}.

Finally, we need an analytic estimate on the behavior of 1-forms on cones. Let $\Xi$ be the smooth part of a convex spherical cone-metric on the 2-sphere with at least three vertices, and $B_{r_0}:=(0, r_0) \times \Xi$ be equipped with the metric of a hyperbolic cone over $\Xi$ (i.e., $g=dr^2+\sinh^2(r)g_{\Xi}$). Again, $(e_1, e_2, e_3)$, $e_3=\pt/\pt r$, is an orthonormal frame, now on $B_{r_0}$. The following is a consequence of Corollary 2.14 and Lemma 3.18 in~\cite{Wei2} and of the fact that on $\Xi_r:=\{r\} \times \Xi$ the volume form is $\sinh^2(r)dvol_{\Xi}$.

\begin{lm}
\label{anest}
Let $\xi$ and $\zeta$ be real-valued 1-forms on $B_{r_0}$ such that they, $d\xi$ and $\delta\xi$ are in $L^2$ and $\Delta \xi + 4\xi=\zeta$ (note that here $d$, $\delta$ and $\Delta$ are the standard differential, codifferential and the Hodge Laplacian for the Riemannian metric on $B_{r_0}$). Then there exists $\alpha>0$ such that as $r \ra 0$ we have
$$\|\xi\|_{L^2(\Xi_r)}=O(r^{\alpha+1}),~~~~~\|\nabla_{e_i}\xi\|_{L^2(\Xi_r)}=O(r^\alpha)~~~~~ {\rm for}~i=1,2.$$
Here the norms are taken for the restrictions of the respective forms to $\Xi_r$.
\end{lm}

Now we have all ready to prove Lemma~\ref{cohomol}.

\begin{proof}[Proof of Lemma~\ref{cohomol}]
Since $y \in \bigcap_e \ker(d\tr_{\gamma_e})$, Lemma~\ref{L_2} implies that $y \in H_{L^2}(L, E)$. By Lemma~\ref{hodge} it can be represented by an $L^2$-harmonic 1-form $\omega \in \Omega^1_{L^2}(L, E)$. Define $L_r:=L\backslash B_r(\Sigma)$ and $K_r:=\pt L_r$. Integrating by parts, we obtain (see~\cite[Proposition 1.3]{HK})
$$\int_{L_r} (|d^E \omega|^2+|\delta^E \omega|^2)=\int_{L_r}(|D\omega|^2+|D^t\omega|^2+\langle Q\omega, \omega \rangle) - \int_{K_r}(*T\omega \wedge \omega + T^t\omega \wedge *\omega),$$
where $Q=TT^t+T^tT$.
Since $\omega$ is harmonic, this expression is equal to zero. We show that the boundary terms
$$A_r:=-\int_{K_r}(*T\omega \wedge \omega + T^t\omega \wedge *\omega)$$
converge to zero as $r \rightarrow 0$. It is shown in~\cite{MM} that $\langle Q \omega, \omega \rangle \geq c |\omega|^2$ for $c>0$. Therefore, we obtain that $\omega$ is zero and hence $y=0$. 

We decompose $K_r$ into the union of $K_{e,r}:=\pt B_r(e)\backslash (\bar B_r(u_1) \cup \bar B_r(u_2))$ where $e$ is an edge of $\Sigma$ with vertices $u_1$ and $u_2$, and of $K_{u,r}:=K_r \cap \pt B_r(u)$ where $u$ is a vertex of $\Sigma$. We denote the respective integrals over $K_{e,r}$ and $K_{u,r}$ by $A_{e,r}$ and $A_{u,r}$. The terms $A_{e,r}$ are bounded exactly as in~\cite[Proposition 3.17]{Wei2} without changes. What is different is the bound on $A_{u,r}$. 

At every point $p \in L$ the fiber $E_p$ can be decomposed into the direct orthogonal sum of the infinitesimal rotations around $p$ and the infinitesimal translations. 
The second summand is naturally isomorphic to $T_pL$. A special feature of dimension 3 is that the first summand is also naturally isomorphic to $T_pL$, see~\cite[Section 2]{HK}. We denote the respective decomposition of $E$ by $E=E^1 \oplus E^2$. Decompose respectively $\omega=(\omega^1, \omega^2)$. Then 
\begin{equation}
\label{t}
A_{u,r}=2\int_{K_{u,r}}\omega^1\wedge\omega^2,
\end{equation} 
see~\cite[Lemma 3.16]{Wei2}. 


For $\gamma \in \mc P_u$ we produce a 1-form $\omega_\gamma \in \Omega^1_{L^2}(\Xi_u, E)$ by the construction described after the statement of Lemma~\ref{cohomol} and denote its class by $y_\gamma \in H^1_{L^2}(\Xi_u, E)$. We also recall the spaces $Y_u \subset H^1_{L^2}(\Xi_u, E)$.
The surface $\Xi_u$ is a deformation retract of $\bar B_{r_0}(u)$ by the radial projection $\pi_u: \bar B_{r_0}(u) \ra \Xi_u$. Consider the image $y_u$ of $y$ in $H^1(\Xi_u, E)$. Because of the assumption on $y$, we have $y_u \in Y_u$. Thereby, it can be represented by a 1-form $\omega_u$, which is a linear combination of the 1-forms $\omega_\gamma$ for $\gamma \in \mc P_u$. Note that $\omega_u$ is in $L^2$. Abusing the notation, we continue to denote the $\pi_u$-pullback of this form to $B_{r_0}(u)$ by $\omega_u$. The computation of $L^2$-norm shows that the $\pi_u$-pullbacks of sections and 1-forms in $L^2$ are in $L^2$, particularly $\omega_u$ is in $L^2$ on $B_r(u)$. Thus, on $B_r(u)$ we have $\omega=d^Es+\omega_u$ for an $L^2$-section $s$.

Substituting this into~(\ref{t}), we decompose $A_{u,r}$ into two terms, $A_{u,r}=A_{u,r}^s+A_{u,r}^\omega$, where $A_{u,r}^\omega:=2\int_{K_{u,r}}\omega^1_u\wedge\omega^2_u$. Let us bound the first summand. Let $s=(s^1, s^2)$ according to $E=E^1 \oplus E^2$. Because $\omega$ is coclosed, we get $\Delta^E s=-\delta^E \omega_u$, where $\Delta^E$ is the Hodge Laplacian of the bundle $E$. We claim that $\delta^E \omega_u$ is in $L^2$. Indeed, the holonomy restricted to $\Xi_u$ preserves a point $p \in \H^3$. The bundle $E$ over $B_{r_0}(u)$ can be endowed with a parallel metric inherited from $T_{p}\H^3$ (this is not the standard metric of $E$). Denote by $\delta^E_0$ the codifferential on $\Omega^\bullet(B_{r_0}(u), E)$ with respect to the new metric. Note that on $\Xi_u$ the section $\delta^E_0 \omega_u$ is compactly supported (because $\omega_u$ on $\Xi_u$ is compactly supported), thus it is in $L^2$. It follows that $\delta^E_0 \omega_u$ on $B_{r_0}(u)$ is also in $L^2$. But the difference $\delta^E-\delta^E_0$ is a bounded operator, see~\cite[p. 353]{Wei2}. Since $\omega_u$ is itself in $L^2$ on $B_{r_0}(u)$, we get that $\delta^E \omega_u$ is in $L^2$ as claimed.

As we mentioned, both sub-bundles $E^1$ and $E^2$ are naturally isomorphic to $TL$. Hence, we may view $s^1$ and $s^2$ as vector fields on $L$.
Denote by $\xi^1$ and $\xi^2$ the real-valued 1-forms dual to to $s^1$ and $s^2$ and by $\zeta^1$ and $\zeta^2$ the real-valued 1-forms dual to the components of $-\delta^E\omega_u$. The equation $\Delta^E s=-\delta^E \omega_u$ translates to (see~\cite[p. 347]{Wei2})
$$\Delta \xi^1 + 4\xi^1=\zeta^1,~~~~~\Delta \xi^2 + 4\xi^2=\zeta^2,$$
where all the terms are in $L^2$. We claim that $d\xi^j$ and $\delta \xi^j$ are also in $L^2$, $j=1,2$. Indeed, since $T$ is bounded, $\nabla\xi^j$ are in $L^2$. But $d\xi^j=\e\circ \nabla\xi^j$ and $\delta\xi^j=-\iota\circ\nabla\xi^j$. Thereby, we can apply Lemma~\ref{anest} and get as $r \ra 0$
\begin{equation}
\label{t2}
\|\xi^j\|_{L^2(\Xi_{u,r})}=O(r^{\alpha+1}),~~~~~\|\nabla_{e_i}\xi^j\|_{L^2(\Xi_{u,r})}=O(r^\alpha),
\end{equation}
where $\Xi_{u,r}=\pt B_r(u)$ and $i=1,2$.

Since $\ad$ switches the infinitesimal rotations and infinitesimal translations, we have $(d^E s)^1=Ds^1+Ts^2$ and $(d^E s)^2=Ds^2+Ts^1$. We substitute it to $A_{u,r}^s$:
$$A_{u,r}^s=A_{u,r}-A_{u,r}^\omega=2\int_{K_{u,r}} \langle (Ds^1+Ts^2+\omega^1_u)(e_1), (Ds^2+Ts^1)(e_2) \rangle-$$ $$-2\int_{K_{u,r}} \langle (Ds^1+Ts^2+\omega^1_u)(e_2), (Ds^2+Ts^1)(e_1) \rangle+$$
$$+2\int_{K_{u,r}} \langle (Ds^1+Ts^2)(e_1), \omega^2_u(e_2) \rangle-2\int_{K_{u,r}} \langle (Ds^1+Ts^2)(e_2), \omega^2_u(e_1) \rangle.$$
We expand this expression into further integrals of the scalar products of the individual summands and bound every integral separately by the Cauchy--Schwarz Inequality. We will compute the $L^2$-norms not over $K_{u,r}$, but rather over $\Xi_{u,r}\supset K_{u,r}$. For terms $\omega^j_u(e_i)$ we use the trivial pointwise bounds $|\omega^j_u(e_i)|=O(r^{-1})$, thus $\|\omega^j_u(e_i)\|_{L^2(\Xi_{u,r})}=O(1)$ as $r \ra 0$. Together with (\ref{t2}) and with the fact that $T$ is bounded it gives $|A_{u,r}^s|=O(r^\alpha)$.

Now it remains to bound $A_{u,r}^\omega$. We have
$$\omega_u=\sum_{\gamma \in \mc P_u} \lambda_\gamma \omega_\gamma$$
for some $\lambda_\gamma \in \C$. Since the supports of $\omega_\gamma$ are disjoints, we get
$$A^\omega_{u,r}=\sum_{\gamma \in \mc P_u} 2\lambda_\gamma \int_{K_{u,r}}\omega^1_\gamma\wedge\omega^2_\gamma$$
(recall that we do not distinguish in notation the forms on $\Xi_u$ and their pullbacks on $B_{r_0}(u)$). Define $A_{\gamma,r}:=\int_{K_{u,r}}\omega^1_\gamma\wedge\omega^2_\gamma.$

For a given $\gamma \in \mc P_u$ the form $\omega_\gamma$ on $\Xi_u$ is supported in $\bar C_\gamma$, which belongs to one of the pair of pants determined by $\mc P_u$ (since the collars are disjoint for different curves). Let us denote this pair of pants by $\Xi'_\gamma$. Denote the respective surface $\Xi(\alpha_1, \alpha_2, \alpha_3)$, given by the suitability of $\mc P_u$, by $\Xi_\gamma$. Hence, after a scaling, $\Xi'_\gamma$ naturally embeds to $\Xi_\gamma$. The spherical metric provides us a holonomy map $\rho_\gamma: \pi_1(\Xi_\gamma) \ra {\rm SO}(2)$. We consider ${\rm SO}(2)$ as a subgroup of ${\rm SL}(2, \C)=\tilde G$, and view $\rho_\gamma$ as having values in $\tilde G$. Thus, we can define the twisted $\Ad\circ\rho_\gamma$-bundle over $\Xi_\gamma$ with the fiber $\mf g$, which we denote by $E_\gamma$. Since such a bundle depends only on the conjugacy class of a representation, the pullback of $E_\gamma$ to $\Xi'_\gamma$ is naturally isomorphic to the restriction of $E$ to $\Xi'_\gamma$. We now denote $E_\gamma$ also just by $E$ and consider the 1-form $\omega_\gamma$ now as a $E$-valued 1-form on $\Xi_\gamma$. We extend $E$ to the bundle over $B_{\gamma,r_0}:=(0, r_0)\times\Xi_\gamma$ equipped with the hyperbolic cone-metric. The radial projection of $B_{\gamma,r_0}$ to $\Xi_\gamma$ is a homotopy equivalence, and we continue to denote the pullback of $\omega_\gamma$ to $B_{\gamma,r_0}$ by $\omega_\gamma$. Define $C_{\gamma, r}$ to be the intersection of the support of $\omega_\gamma$ with $\{r\}\times\Xi_\gamma$.

Since $\omega_\gamma$ is compactly supported on $\Xi_\gamma$, it is in $L^2$, both on $\Xi_\gamma$ and on $B_{\gamma, r_0}$. Due to~\cite[Lemma 3.9]{Wei2}, we have $H^1_{L^2}(\Xi_\gamma, E)=0$. Thereby, we have $\omega_\gamma=d^Es$ for a smooth $L^2$-section $s$ on $\Xi_\gamma$, which we extend to $B_{\gamma, r_0}$. We get $\Delta^E s=\delta^E\omega_\gamma$ on $B_{\gamma, r_0}$. Pick an orthonormal frame $(e_1, e_2, e_3)$, $e_3=\pt/\pt r$, on $B_{\gamma, r_0}$. We have
$$A_{\gamma, r}=\int_{C_{\gamma,r}} \langle (Ds^1+Ts^2)(e_1), (Ds^2+Ts^1)(e_2) \rangle-$$ $$-2\int_{C_{\gamma,r}} \langle (Ds^1+Ts^2)(e_2), (Ds^2+Ts^1)(e_1) \rangle.$$
Using Lemma~\ref{anest}, we obtain that $|A_{\gamma, r}|=O(r^{2\alpha})$ as $r \ra 0$. We conclude that $A_r\ra 0$, which finishes the proof.
\end{proof}

\section{Properness}
\label{psec}

The goal of this section is to prove

\begin{lm}
\label{proper}
Let $\{g_i\}$ be a sequence in $\mathcal P(M, V)$ such that the sequence of the induced dual metrics $\{d_i\}$ converges to a metric $d \in \ol{\mathcal D}_M(\partial M, V)$. Then, up to extracting a subsequence, $\{g_i\}$ converges to $g \in \ol{\mathcal P}(M, V)$ with $V(g)=V(d)$.
\end{lm}

We will employ the \emph{spherical model} of $\H^3$ and $\dS^3$. Denote the \emph{Euclidean scalar product} on $\R^4$ by $\langle .,. \rangle_E$, i.e.,
$$\langle x, y \rangle_E :=x_0y_0+x_1y_1+x_2y_2+x_3y_3.$$
We denote by $\S^3 \subset \R^4$ the standard unit sphere
$$\S^3=\{x \in \R^4: \langle x, x \rangle_E =1\}$$
endowed with the metric induced from $\langle .,. \rangle_E$. The central projection from the origin of $\R^4$ maps injectively $\H^3$ and $\dS^3$ to subsets of $\S^3$ denoted by $\H_s^3$ and $\dS_s^3$. Notably it sends geodesics to geodesics, and, thus, convex sets to convex sets. The boundary at infinity $\partial_{\infty} \dS^3$ projects to two spherical caps denoted by $\partial_{\infty}^+ \dS_s^3$ and $\partial_{\infty}^- \dS_s^3$.

\begin{rmk}
If $C \subset \H_s^3$ or $C \subset \dS_s^3$ is (future-)convex, then the dual set $C^*$ can be obtained from the usual polarity on $\S^3$ composed with the reflection with respect to the hyperplane $x_0=0$. Indeed, this follows from the fact that the linear plane, orthogonal to a vector with respect to the Minkowski metric, differs from the linear plane, orthogonal to the same vector with respect to the Euclidean metric, by reflection with respect to the hyperplane $x_0=0$.
\end{rmk}

Let $S=S_j$ be a connected component of $\partial M$ and $W=V_j \subset V$ be the subset of $V$ marking the faces of $S$. (To simplify the notation, for the moment we will omit the subscript $j$.) We denote by $\nu=\nu_j: \pi_1(S) \rightarrow \pi_1(M)$ the push-forward homomorphism induced by the inclusion of $S$ to $M$, and denote by $D=D_j$ the covering space of $S$ corresponding to $\ker(\nu)$. The latter is homeomorphic to a  boundary component of $\tilde M$ (possibly to infinitely many of them). It is either homeomorphic to the open 2-disk if $S$ is incompressible, otherwise $D$ is infinitely connected. 

Denote by $d \in \mc D(S, W)$ the restriction of a metric from $\mc D_M(\pt M, V)$. We continue to denote the lift of $d$ to $D$ by $d$. Note that all closed geodesics in $(D, d)$ have lengths greater than $2\pi$. Let $\mathcal T$ be a triangulation of $(S, W)$ such that $d$ is $\mc T$-triangulable and its lift to $D$ is a simplicial complex. We continue to denote this lift by $\mathcal T$, but we denote the lift of $W$ to $D$ by $\tilde W$.

\begin{dfn}
\label{regemb}
An isometric embedding $\phi: (D , d) \rightarrow \dS^3$ is called \emph{$\mathcal T$-regular} if

1) it is polyhedral with respect to $\mathcal T$,

2) it is future-convex, i.e., the image belongs to the boundary of a future-convex set.
\end{dfn}

We assume that $\phi$ maps to the spherical model $\dS_s^3 \subset \S^3$. Since $\phi$ is an isometric embedding, all faces are necessarily space-like. 
From future-convexity, $\phi(D ) \subset \conv_s(\phi(D ))$, where $\conv_s$ is the closed convex hull in $\S^3$, though a priori it is possible that $\conv_s(\phi(D )) \cap \dS^3$ is not future-convex.

But this is not the case for us, because the metric $d$ is complete. Indeed, we claim that all limit points of $\phi(D )$ that do not belong to $\phi(D )$ are in $\partial_{\infty}^+ \dS^3$. Suppose the contrary, that such a point $q$ exists in $\dS^3$. Then $q \in \pt\conv_s(\phi(D ))$. Pick $p \in \phi(D )$, there exists a rectifiable curve (with respect to the spherical metric) $\gamma:[0, a_0] \ra \pt\conv_s(D ) \cap \dS^3$ connecting $p$ and $q$. Let $[0, a]$ be the maximal segment such that $\gamma|_{[0, a)}$ belongs to $\phi(D )$. The curve $\gamma|_{[0, a)}$ has finite spherical length and belongs to a space-like polyhedral surface, hence it also has a finite de Sitter length. Since $\phi(D )$ is relatively open in $\pt\conv_s(\phi(D ))$, the point $\gamma(a)$ is not in $\phi(D )$. But $\phi^{-1}\circ\gamma|_{[0, a)}$ is a curve of finite length in $(D , d)$, hence, due to completeness of $d$, it extends to $a$. Thereby, $\gamma(a) \in \phi(D )$, which is a contradiction. It follows that $\conv_s(\phi(D ))\cap \dS^3_s$ is a future-convex set and
\begin{equation*}
\big(\partial\conv_s(\phi(D )) \backslash \phi(D )\big) \subset \big(\partial_{\infty}^+ \dS_s^3 \cup \H_s^3\big).
\end{equation*}

In the setting of Lemma~\ref{proper} consider the sequence of $\tilde M(g_i)$ embedded to $\H^3$ equivariantly with respect to the holonomy maps $\rho_i$. There exists a connected component of $\pt\tilde M(g_i)$ that is invariant with respect to $\rho_i|_{\im(\nu)}$. The dual surface to this component is an embedding $\phi_i: D \ra \dS^3$. We first study what happens with this sequence of embeddings provided that the induced metrics converge (which are the dual metrics for the respective components of $\pt \tilde M(g_i)$). We now focus on the restrictions of $d_i$ and $d$ to $S$, so for the moment we perceive them as elements of $\mc D(S, W)$.  

Due to Lemma~\ref{fintriang}, there are only finitely many triangulations $\mc T$ such that one of $d_i$ is $\mc T$-triangulable. By passing to a subsequence, we may assume that there is a triangulation $\mathcal T$ such that its lift to $D $ is a face triangulation of all $\phi_i(D )$, particularly all $d_i$ are $\mc T$-triangulable. Note that the lift of $\mathcal T$ to $D $ is a simplicial complex and all $\phi_i$ are $\mathcal T$-regular.

\begin{lm}
\label{dualconv}
After possibly composing each $\phi_i$ with an isometry $\eta_i$ of $\dS^3$ and passing to a subsequence, $\phi_i$ converge to a $\mathcal T$-regular embedding $\phi: (D , d) \rightarrow \dS^3$.
\end{lm}

\begin{proof}
After passing to a subsequence, we can assume that for all $v \in \tilde W $ the sequence $\phi_i(v)$ converges to a point in $$\overline\dS_s^3:=\dS_s^3 \cup \partial_{\infty}^+ \dS_s^3 \cup \partial_{\infty}^- \dS_s^3.$$ We denote this point by $\phi(v)$, so $\phi$ is a map from $\tilde W $ to $\overline\dS_s^3$.

Note that due to the future-convexity of $\phi_i$, the intersection $\phi(\tilde W ) \cap \partial_{\infty}^- \dS_s^3$ consists of at most one point. If this intersection is a point $p$, then we can compose each $\phi_i$ with an isometry $\eta_i$ of $\dS^3$ that has $p$ as its fixed repulsive point so that this intersection becomes empty. Abusing the notation, we continue to denote the corrected maps $\eta_i \circ \phi_i$ just by $\phi_i$. By a similar procedure we can also assume that $\phi(\tilde W )\cap \dS_s^3$ consists of at least one point, continuing to keep $\phi(\tilde W ) \cap \partial_{\infty}^- \dS_s^3$ empty. Let us call the points of $\tilde W $ with $\phi$-images in $\partial_{\infty}^+ \dS_s^3$ \emph{degenerate}. One of the key steps in our proof is to establish that there are no degenerate points.

Observe that if $v,w \in \tilde W $ are degenerate and adjacent in $\mathcal T$, then $\phi(v)=\phi(w)\in \partial_{\infty}^+\dS_s^3$, as otherwise for some large $i$ the segment between $\phi_i(v)$ and $\phi_i(w)$ is not space-like.

We now extend $\phi$ to a continuous $\mathcal T$-polyhedral map sending all $D $ to $\dS_s^3 \cup \partial_{\infty}^+ \dS_s^3$. For a triangle $T$ without non-degenerate vertices, we just extend $\phi$ to $T$ isometrically. The same we do for every remaining edge with both non-degenerate vertices. For an edge or a triangle with all degenerate vertices, we just extend $\phi$ by sending all the edge or all the triangle to the image of the vertices. For edges that have both degenerate and non-degenerate vertices, we must send them to light-like segments, and there is no canonical way to do this, so we do this arbitrarily. After this we arbitrarily extend the map to the triangles that have both degenerate and non-degenerate vertices, by sending their interiors to the relative interiors of the linear spans of the vertices.

This way $\phi$ becomes a continuous $\mathcal T$-polyhedral map from $D $ to $\dS_s^3 \cup \partial_{\infty}^+ \dS_s^3$. We note that by continuity, we still have that the image of $\phi$ belongs to the closure of the boundary of a future-convex set. We will refer to this as to the future-convexity of $\phi$.


For $z \in (D  \backslash \tilde W )$ belonging to the interior of a triangle of $\mathcal T$, we define the open star $\ost(z)$ of $z$ to be the interior of this triangle. If $z$ belongs to the relative interior of an edge of $z$, then by $\ost(z)$ we denote this relative interior together with the interiors of both triangles adjacent to this edge. When $z \in \tilde W $, by $\ost(z)$ we mean the relative interiors of all edges and triangles adjacent to $z$ together with $z$ itself. The open star $\ost(U)$ of a set of points $U \subset D $ is defined as the union of $\ost(z)$ for all $z\in U$.

Define $X:=D \backslash \ost(\phi^{-1}(\partial_{\infty}^+\dS_s^3))$. By construction, $X$ is non-empty. By continuity, the restriction of $\phi$ to $X$ is a path-isometry. We claim that the restriction of $\phi$ to $X$ is injective. Indeed, since it is a path-isometry, it is locally injective. Now suppose that for two points $x, y \in X$ we have $\phi(x)=\phi(y)$. Due to local injectivity, $x \notin \ost(y)$ and vice versa. Pick some small $r$ such that for all sufficiently large $i$ the $r$-neighborhood of $y$ in $(D , d_i)$ is contained in $\ost(y)$. As $\phi(x) \in \dS^3_s$, we can pick its neighborhood $U$ such that (1) for every point $p \in U$ in space-like relation to $\phi(x)$ the future-cones of $\phi(x)$ and of $p$ intersect and so do their past-cones, and (2) the length of any space-like segment in $U$ is at most $r/2$. Next, for every $i$ we pick a neighborhood $U_i$ of $\phi_i(x)$ converging to $U$, satisfying the same conditions, except that in (2) we replace $r/2$ with $r$. As $\phi(x)=\phi(y)$, for large enough $i$ we get $\phi_i(y) \in U_i$. For $p \in \dS^3$ by $FP(p)$ denote the union of the closures of its future- and past-cone. Then $\phi_i(\ost(y))$ intersects $FP(\phi_i(x))$. Indeed, by (1) the space-like cone from $\phi_i(y)$ spanned by $\phi_i(\ost(y))$ intersects $FP(\phi_i(x))$, and by (2) there is an intersection point belonging to $\phi_i(\ost(y))$. This is a contradiction with $\phi_i$ being future-convex.

Suppose that there exists $o \in \phi(D) \cap \pt^+_\infty \dS^3_s$. Then $\phi^{-1}(o)$ is a closed $\mc T$-simplicial set. Define $Y:=\ost(\phi^{-1}(o))$. By construction, $\pt Y\subset X$, hence, the restriction of $\phi$ to $\pt Y$ is an injective path-isometry. 
Let $\Pi$ be the tangent plane to $\pt_\infty^+\dS^3_s$ at $o$ and $P:=\phi(\ol Y)$, where $\ol Y$ is the closure of $Y$ in $D$. Due to future-convexity, we have $P \subset \Pi$.
We claim that for any $x, y \in \pt Y$ the segments between $o$ and $\phi(x)$, $\phi(y)$ are disjoint. A proof is similar to the one above. Indeed, suppose the contrary, and suppose that $\phi(x)$ is between $o$ and $\phi(y)$. Once again, $\ost(x)$ and $\ost(y)$ are disjoint. We pick $r$ such that for all sufficiently large $i$ the $2r$-neighborhood of $y$ in $(D, d_i)$ is contained in $\ost(y)$ and pick neighborhoods $U$ of $\phi(x)$ and $U_i$ of $\phi_i(x)$ as above. Next, consider a vertex $v \in Y \cap \ol{\ost}(y)$, which exists since $y \in \pt Y$. We have $\phi(v)=o$. Pick a sequence $y_i \in D$ such that $\phi_i(y_i)$ belongs to the segment between $\phi_i(v)$ and $\phi_i(y)$, and $\phi_i(y_i)$ converge to $\phi(x)$. Then $y_i$ converge to $y$. For large enough $i$ the $r$-neighborhoods of $y_i$ in $(D, d_i)$ are contained in $\ost(y)$. Next, for $y'$ in the $r$-neighborhood of $y_i$, the space-like distance between $\phi_i(y')$ and $\phi_i(y_i)$ is at least the intrinsic distance between $y'$ and $y_i$. Also for large enough $i$ we have $\phi_i(y_i) \in U_i$, thus $\phi_i(\ost(y_i))$ intersects $FP(\phi_i(x))$. Because $\ost(y_i)\subset\ost(y)$, $\phi_i(\ost(y))$ intersects $FP(\phi_i(x))$, which is again a contradiction.

It follows that the degree of every vertex in $\pt Y$ is at most two, where the degree is meant in $\pt Y$. It also follows that there are two options: either $\pt Y$ consists of non-closed simple paths, or from one simple closed curve.

In the first case, we note that since the restriction of $\phi$ to $\pt Y$ is an injective path-isometry, and the length of $\phi(\pt Y)$ is finite, then the length of each component of $\pt Y$ is finite. Therefore, there exists a simple closed edge-path $\gamma \subset \phi^{-1}(o)$ bounding a component of $\pt Y$. Pick a vertex $v$ in this component. Consider a time-like ray $R$ from $\phi(v) \in \dS^3_s$ and a sequence of time-like rays $R_i$ from $\phi_i(v)$ converging to $R$. Due to future-convexity, $R_i$ intersects $\conv_s(\phi_i(\gamma))$, but this is not true for $R$ and $\conv_s(\phi(\gamma))=o$. This is a contradiction.

In the second case, we have $P = \conv_s(\phi(D )) \cap \Pi$ (this is because $\Pi$ cannot contain the $\phi$-image of any triangle not from $\ol Y$). In particular, $P$ is a compact convex set. As $\partial P$ is a convex curve in $\dS_s^3$ around a point at infinity in a degenerate plane $\Pi$, it has length $2\pi$. As the restriction of $\phi$ to $\pt Y$ is a path-isometry, the latter also has length $2\pi$. One can also see that its angles on the side to $Y$ are $\pi$. This is shown in~\cite[Lemma 6.6]{Sch}. We quickly sketch the argument. Take a vertex $p$ of $P$ and consider a small enough triangular neighborhood $A$ of $p$ in $P$. Let $v:=\phi^{-1}(p)$. We denote by $A_i$ triangular neighborhoods of $\phi_i(v)$ in $\phi_i(\ol Y)$ converging to $A$. The sets $\phi_i^{-1}(A_i)$ are triangular neighborhoods of $v$ in $\ol Y$ with metrics $d_i$. Up to taking a subsequence, they converge to a triangle $B \ni v$ in $\ol Y$ with the metric $d$. But as $A$ is metrically a degenerate triangle, $B$ is a geodesic segment on the side of $\ol Y$ (having angle $\pi$ at $v$). As all cone angles of $d$ are greater than $2\pi$, $\partial Y$ is a simple closed geodesic of length $2\pi$ in $(D , d)$. This is a contradiction.

Thus, $X=D$, so $\phi(D) \subset \dS_s^3$ and $\phi$ is a $\mc T$-polyhedral future-convex isometric embedding of $(D, d)$, i.e, it is $\mc T$-regular. 
\end{proof}

Each embedding $\phi_i$ is equivariant with respect to a convex cocompact, discrete and faithful representation $\sigma_{i}:=\rho_i|_{\im(\nu)}: \im(\nu ) \rightarrow G$. As $\phi_i$ converge to $\phi$, for each $\gamma \in \pi_1(S )$ the elements $\sigma_{i}(\gamma)$ converge to an element of $G$. A theorem of J{\o}rgensen~\cite{Jor} says that the set of discrete and faithful representations to $G$ is closed in the space of representations. Thus, $\sigma_{i}$ converge to a discrete and faithful representation $\sigma : \im(\nu ) \rightarrow G$ and $\phi$ is equivariant with respect to $\sigma $.

The dual to the embedded surface $\phi(D )$ is a convex polyhedral embedding of $D $ to $\H^3$ equivariant with respect to $\sigma $. As ${\rm im}(\nu )$ contains no torsion, the image of $\sigma $ contains no elliptic elements, so the action is free. By taking the quotient, we obtain a hyperbolic convex cone-metric on $S $, which we denote by $d^I$. By continuity, we get that the metrics $d_i^I$ on $S $ induced by $g_i$ at the boundary of $M$ converge uniformly to a convex hyperbolic cone-metric $d^I$ on $S $.

We do these procedures for each connected component $S_j$ of $\partial M$. We obtain a convex hyperbolic cone-metric on $\partial M$, which we continue to denote by $d^I$.

\begin{lm}
\label{compactind}
Let $\{g_i\}$ be a sequence in $\mathcal P(M, V)$ such that the sequence of the induced intrinsic metrics $\{d_i^I\}$ on $\partial M$ converges uniformly to a convex hyperbolic cone-metric $d^I$ on $\partial M$. Then up to extracting a subsequence, the sequence $\{\overline g_i\}$ converges to a convex cocompact hyperbolic metric $\overline g \in \mathcal{CH}(N)$.
\end{lm}

We need a short preliminary work. Let $\overline g$ be a convex cocompact hyperbolic metric on $N$ and $s \in \partial_{\infty}\H^3 \backslash \Lambda(\overline g)$. Consider a small horosphere at $s$ such that it does not intersect $\tilde C(\overline g)$. Let $\mu(s) \in \partial \tilde C(\overline g)$ be the closest point from $\tilde C(\overline g)$ to this horosphere. Clearly, $\mu(s)$ is independent on the choice of a horosphere, and determines a map from $\partial_{\infty} N(\overline g)$ to $\partial C(\overline g)$, still denoted by $\mu$. 
When $\overline g$ is Fuchsian of the first kind, we modify $\mu$ so that its target is two copies of $C(\overline g)$, and $\mu$ remembers, from which side of $C(\overline g)$ the point of $\partial_{\infty}N(\overline g)$ was located. When $\overline g$ is Fuchsian of the second kind, we consider the double of $C(\overline g)$ along its relative boundary as the target of $\mu$ (keeping track on the side where $s$ was located except when $\mu(s)$ is in the relative boundary of $C(\overline g)$). 

The Uniformization Theorem endows $\partial_{\infty}\H^3 \backslash \Lambda(\overline g)$ with a conformal $\rho_{\overline g}$-invariant hyperbolic metric, projecting to a conformal hyperbolic metric on $\partial_{\infty} N(\overline g)$. The following statement compares this metric structure with the induced intrinsic metric on $\partial C(\overline g)$:

\begin{thm}
\label{su}
There exist functions $a, b: (0, +\infty)\rightarrow (0, +\infty)$ such that if for $r_0>0$ the injectivity radius of $\partial C(\overline g)$ is at least $r_0$, then $\mu$ is $a(r_0)$-Lipschitz and has a $b(r_0)$-Lipschitz homotopy inverse. If $\partial M$ is incompressible, then $a,b$ can be chosen absolute constants.
\end{thm}

It was first sketched by Sullivan in~\cite{Sul} for the case of incompressible boundary, a more detailed proof can be found in~\cite{EM}. The case of general boundary was obtained in~\cite{BC}.

Finally, we need a well-known compactness criterion for subsets of the Teichm\"uller space (complementing the \emph{Mumford Compactness Theorem}).

\begin{thm}
\label{flp}
Let $S$ be a closed oriented surface. There exists a finite set of isotopy classes of simple closed curves $\{\gamma_1, \ldots, \gamma_l\}$ such that for any positive number $C$ the set of hyperbolic metrics, having geodesic representatives of all $\gamma_i$ with lengths at most $C$, is compact in the Teichm\"uller space $\mc{TS}(S)$ of $S$.
\end{thm}

It follows straightforwardly from Lemma 7.10 and Proposition 7.11 in~\cite{FLP}.

\begin{proof}[Proof of Lemma~\ref{compactind}.]
Because $d_i^I$ converge uniformly to $d^I$, the length of a smallest curve in each free homotopy class of closed curves on $\partial M$ is uniformly bounded from above in $\{d_i^I\}$. Due to the hyperbolic Busemann--Feller Lemma~\cite{Mil}, this implies that they are also uniformly bounded from above in $\partial C(\overline g_i)$. Theorem~\ref{flp} implies that the injectivity radii of $\partial C(\overline g_i)$ are uniformly bounded from below. Due to Theorem~\ref{su}, the length of a smallest curve in each free homotopy class is also uniformly bounded in the hyperbolic representatives of conformal structures in $\partial_{\infty}N(\overline g_i)$. Due to Theorem~\ref{ab} and Theorem~\ref{flp}, this means that $\overline g_i$ stay in a compact subset of $\mathcal{CH}(N)$, and thus, up to taking a subsequence, converge to $\overline g \in \mathcal{CH}(N)$.
\end{proof}

Now we can finish the proof of Lemma~\ref{proper}.

\begin{proof}
Let $v \in V$ and $f_{v,i}$ be the immersed plane in $N(\overline g_i)$ containing the face of $M(g_i)$ marked by $v$. We denote by $h_{v, i}$ the distance from $f_{v, i}$ to $C(\overline g_i)$. We see that $h_{v, i}$ is bounded from above. Indeed, otherwise the diameters of the induced intrinsic metrics on $\partial M(g_i)$ grow to infinity, which contradicts to the fact that they converge to a cone-metric.

Thus, up to taking a subsequence, all $f_{v,i}$ converge to some immersed planes $f_v$ in $N(\overline g) \backslash \inter(C(\overline g))$. Hence, polyhedral surfaces $\partial M(g_i)$ converge to polyhedral surfaces in $N(\overline g)$ bounding together a totally convex set. It is compact, as otherwise, again, the diameters of $d_i^I$ would grow to infinity. Thus, it is homeomorphic to $M$, and we denote it by $M(g)$.

We need to show that $f_{v, i}$ are bounded from below. To see this, let us pass to the universal cover. The surface of $\partial \tilde C(\overline g)$ consists of totally geodesic pieces bounded by lines, which extend to the boundary at infinity. Hence, if a face of $\partial \tilde M(g)$ falls to $\partial \tilde C(\overline g)$, then this face contains a line in $\H^3$. Let $\tilde v \subset \partial M(g)^*$ be the vertex dual to this face. Then there are faces of $\tilde M(g)^*$ adjacent to $\tilde v$ that are touching $\partial_{\infty}^+ \dS^3$. The proof of Lemma~\ref{dualconv} shows that this can not happen. Then $g_i$ converge to $g \in \ol{\mathcal P}(M, V)$. It is evident that $V(g)=V(d)$.



\end{proof}

\section{Final steps}
\label{fsec}

\subsection{More topological preliminaries}
\label{bundlesec}

To finish the proof, we need to examine connectivity properties of the spaces $\mc P(M, V)$ and $\mc D_M(\pt M, V)$. It turns out that our purposes require more machinery.
Denote the forgetful map $\mc{CH}(N, V) \ra \mc{CH}(N)$ by $\sigma_{V}$. The first goal of this section is to show that it is a fiber bundle of a special type. We need to introduce yet another space. 

Let $\mc{PCH}(N)$ be the space of \emph{pointed} convex cocompact hyperbolic metrics on $N$, i.e., the space of pairs $(\ol g, p)$ where $\ol g \in \mc {CH}(N)$ and $p \in N(\ol g)$ is a marked point, considered up to isometry sending the marked point to the marked point and isotopic to the identity. Denote by $\mc{PCH}_j(N)$ the subset of $\mc{PCH}(N)$ where $p$ belongs to the geometric end $E_j$ corresponding to the $j$-th component $S_j$ of $\pt M$. First, we introduce the structure of a fiber bundle on $\mc{PCH}_j(N)$ with respect to the forgetful map $\sigma_j: \mc{PCH}_j(N) \ra \mc{CH}(N)$. For $\ol g \in \mc {CH}(N)$ denote by $\psi_j: E_j \ra \R_+$ the distance from the marked point to the convex core (so this function depends on $\ol g$, but we will suppress this in notation). It is $C^1$ (\cite[Lemma 1.3.6]{EM}) and the level surfaces are strictly convex.

\begin{lm}
\label{fbundle}
There exists a fiber bundle atlas $\{(U, \phi_{U})\}$, $U \subset \mc{CH}(N)$, for $$\sigma_j: \mc{PCH}_j(N) \ra \mc{CH}(N)$$ with
$$\phi_{U}: \sigma^{-1}_j(U) \rightarrow U \times S_j \times \R_+$$
such that \\
(1) for every $\ol g \in U, r \in \R_+,$ the set $\phi_{U}^{-1}(\ol g, S_j, r)$ is the $r$-level surface of $\psi_j$ in $E_j \subset N(\ol g)$;\\
(2) for every $\ol g \in U, x \in S_j,$ the set $\phi_{U}^{-1}(\ol g, x, \R_+)$ is a gradient line of $\psi_j$.
\end{lm}

\begin{proof}
Pick $\ol g \in \mc{CH}(N)$, fix a developing map $\tilde N (\ol g) \cong \H^3$ and the respective holonomy map $\rho$. Recall that $\nu_j: \pi_1(S_j) \ra \pi_1(M)$ is the push-forward homomorphism and $D_j$ is the covering space of $S_j$ with respect to $\ker(\nu_j)$. We denote by $\rho_j$ the restriction of $\rho$ to $\im(\nu_j)$. It fixes a geometric end $\tilde E$ of $\tilde N(\ol g)$ corresponding to a lift of $S_j$. Consider an embedding $D_j \times [1/2, 3/2] \ra \tilde E$, where $D_j \times \{r\}$ is embedded as the surface at distance $r$ from $\tilde C(\ol g)$. These surfaces are strictly convex.

By the Ehresmann--Thurston Theorem~\cite[Theorem 1.7.1]{CEG}, there exists a small neighborhood $\tilde U_j$ of $\rho_j$ in $\mc R(\im(\nu_j), G)$ and a continuous map \[\hat\dev: \tilde U_j \times D_j \times [1/2, 3/2] \ra \H^3\] such that for every $\rho'_j \in \tilde U_j$ and $r \in \R_+$ the map $\hat\dev(\rho'_j, ., r)$ is $\rho'_j$-equivariant. We are actually interested only in its restriction to $\tilde U_j \times D_j \times \{1\}$, so we abuse the notation and denote this restriction by $\hat \dev: \tilde U_j \times D_j \ra \H^3$. Provided that $\tilde U_j$ is sufficiently small, the surfaces $\hat\dev(\rho'_j, D_j)$ are strictly convex. We extend $\tilde U_j$ to an arbitrary neighborhood $\tilde U$ of $\rho$ in $\mc R(\pi_1(M), G)$ (i.e., the projection of $\tilde U$ to $\mc R(\im(\nu_j), G)$ is $\tilde U_j$), and extend $\hat \dev$ trivially to $\tilde U$. Thereby, now we consider $\hat \dev$ as a map from $\tilde U \times D_j$ to $\H^3$. We assume that $\tilde U$ is sufficiently small so that it contains only holonomies of convex cocompact metrics on $N$.

For $\rho' \in \tilde U$ we denote by $\tilde C(\rho')$ the convex hull of its limit set. Provided that $\tilde U$ is sufficiently small, all surfaces $\hat\dev(\rho', D_j)$ are at positive distance from $\tilde C(\rho')$. We denote by $\tilde E(\rho')$ the connected component of $\H^3 \backslash \tilde C(\rho')$ containing this surface, and by $A$ we denote the set of pairs $(\rho', p): \rho' \in \tilde U, p \in \tilde E(\rho')$. Consider the function $\psi_A: A \ra \R_+$ sending $(\rho', p)$ to the distance from $p$ to $\tilde C(\rho')$, and the map $A \ra \H^3$ sending $(\rho', p)$ to the closest point on $\tilde C(\rho')$ to $p$. We claim that these maps are continuous. Indeed, when a sequence $\rho_i$ converges to $\rho'$ in $\tilde U$, the sets $\tilde C(\rho_i)$ converge to $\tilde C(\rho')$ in the Hausdorff sense, see~\cite[Lemma 3.27]{Pro3}. This easily implies the claim. We now define a map $\dev: \tilde U \times D_j \times \R_+ \ra \H^3$ that sends $(\rho', x, r)$ to the point on the same gradient line of $\psi_A(\rho', .)$ as $\hat\dev(\rho', x)$ and that is at distance $r$ from $\tilde C(\rho')$. By our claim, this map is continuous. 

Now we consider the map ${\rm id} \times \dev: \tilde U \times D_j \times \R_+ \ra \tilde U \times \H^3$. It is a homeomorphism onto $A$.  (Here we rely on the strict convexity of the surfaces $\hat\dev(\rho', D_j)$, as it implies that every gradient line of $\psi_A(\rho', .)$ intersects $\hat\dev(\rho', D_j)$ exactly once.) It is equivariant with respect to the fiberwise actions of $\im(\nu_j)$ on $D_j$ and $\H^3$. Denote by $U$ the projection of $\tilde U$ to $\mc X(\pi_1(M), G)$. The map ${\rm id} \times \dev$ projects to a homeomorphism from $U \times S_j \times \R_+$ onto $\sigma^{-1}_j(U)$. The inverse of this homeomorphism is a chart of the desired form.
\end{proof} 

By Theorem~\ref{ab}, $\mc{CH}(N)$ is contractible, hence the bundle from Lemma~\ref{fbundle} admits a trivialization
\begin{equation}
\label{triv_j}
\mc{PCH}_j(N) \cong \mc{CH}(N) \times S_j \times \R_+.
\end{equation}
In what follows we fix such a trivialization.

Now suppose that $V=\{v\} \subset S_j$. The space $\mc{PCH}_j(N)$ is naturally homeomorphic to $\mc{CH}(N, V)$: we send the pair $(\ol g, p) \in \mc{PCH}_j(N)$ to the unique immersed positive plane tangent to the level surface of $\psi_j$ at $p$. Moreover, the composition of $\sigma_V$ with this homeomorphism is $\sigma_{j}$. Hence, Lemma~\ref{fbundle} establishes a fiber bundle structure on $\mc {CH}(N, V)$ with respect to $\sigma_{V}$. 

We proceed to the case of general $V$. Define $V_j:=V \cap S_j$, $n_j:=|V_j|$. We denote by $S_{V}$ the product of $S_j^{V_j}$ where $j=1\ldots m$ and $m$ is the number of components of $\pt M$. Using the natural homeomorphisms $\mc{CH}(N, \{v\}) \cong \mc{PCH}_j(N)$, we endow $\mc{CH}(N, V)$ with the structure of a fiber bundle over $\mc{CH}(N)$ that is the product of $n_j$-th powers of the bundles $\mc{PCH}_j(N)$ where $j=1\ldots m$. The trivializations (\ref{triv_j}) induce a trivialization
\begin{equation*}
\mc{CH}(N, V) \cong \mc{CH}(N) \times S_{V} \times \R_+^V.
\end{equation*}

Now our goal is to construct a covering $\mc P^\sharp(M, V)$ of $\mc P(M, V)$, over which we will define the dual metric map $\mc I^\sharp_V$ with values in $\mc D^\sharp_M(\pt M, V)$. The elements $\mc P^\sharp(M, V)$ will be basically those of $\mc P(M, V)$, but equipped with an additional piece of data, which will allow us to construct a homeomorphism from the ``dual boundary'' onto $(\pt M, V)$.

We denote by $S_j^{\ast V_j}$ the subset of $S_j^{V_j}$ consisting of the injective maps $V_j \hookrightarrow S_j$. Denote by $S_{\ast V}$ the product of $S_j^{\ast V_j}$. Denote by $\mc{CH}^\ast(N, V)$ the subspace of $\mc{CH}(N, V)$ corresponding to the configurations where no two marked planes have their closest points to $C(\ol g)$ on the same gradient line of $\psi_j$. We have a trivialization
\begin{equation}
\label{triv}
\mc{CH}^\ast(N, V) \cong \mc{CH}(N) \times S_{\ast V} \times \R_+^V.
\end{equation}
It is easy to see that $\ol{\mc P}(M, V)$ belongs to $\mc{CH}^\ast(N, V)$.

Finally, denote by $\mc{CH}^\sharp(N, V)$ the universal cover of $\mc{CH}^\ast(N, V)$. The trivialization (\ref{triv}) produces a trivialization
\begin{equation}
\label{trivtilde}
\mc{CH}^\sharp(N, V) \cong \mc{CH}(N) \times \tilde {S_{\ast V}} \times \R_+^V.
\end{equation}
Denote by $\mc P^\sharp(M, V)$ and $\ol{\mc P}^\sharp(M, V)$ the coverings of $\mc P(M, V)$ and $\ol{\mc P}(M, V)$ respectively in $\mc{CH}^\sharp(N, V)$.

Our goal now is to describe a map $\mc I^\sharp_V: \ol{\mc P}^\sharp(M, V) \ra \ol{\mc D}_M^\sharp(\pt M, V)$ lifting the map $\mc I_V$. Denote by $H_0(\pt M)$ the space of self-homeomorphisms of $\pt M$ isotopic to the identity, endowed with the compact-open topology. Recall that $H_0^\sharp(\pt M, V)$ is its subgroup consisting of self-homeomorphisms isotopic to the identity by an isotopy fixing $V$. The inclusion $V \hookrightarrow \pt M$ may be considered as an element of the configuration space $S_{\ast V}$. Fix its lift $\hat V$ to the universal cover $\tilde{S_{\ast V}}$. Every $h \in H_0(\pt M)$ determines a self-homeomorphism of $S_{\ast V}$, which, thanks to the homotopy lifting, determines a self-homeomorphism of $\tilde{S_{\ast V}}$. By considering its evaluation at $\hat V$, we define a map $H_0(S) \ra \tilde{S_{\ast V}}$, which factors through $H_0^\sharp(S, V)$ to the evaluation map $${\rm ev}: H_0(\pt M)/H_0^\sharp(\pt M, V) \ra \tilde{S_{\ast V}}.$$ It was shown in~\cite[Lemma 3.1]{FP} that it is a homeomorphism. We fix such a map (which depends only on the choice of $\hat V$). 

Let $\ol g \in \mc{CH}(N)$ and $\Sigma \subset N(\ol g)$ be a separating convex surface in the geometric end $E_j$ corresponding to $S_j$ (particularly, $\Sigma$ is at a positive distance from $C(\ol g)$). The set of the closest points to $C(\ol g)$ on the supporting planes to $\Sigma$ form another surface $\Sigma' \subset E_j$. By construction, no two points of $\Sigma'$ belong to the same gradient line of $\psi_j$. Thereby, trivialization (\ref{triv_j}) provides a homeomorphism $\chi: \Sigma' \ra S_j$. The dual metric of $\Sigma$ may be considered as a metric on $\Sigma'$. When some supporting planes to $\Sigma$ are marked by a set $V_j$, the set $V_j$ respectively mark points on $\Sigma'$. 

Pick now $g^\sharp \in \ol{\mc P}^\sharp(M, V)$,  let $(g, f) \in \ol{\mc P}(M, V)$ be its projection, and $\Sigma$ be the $j$-th component of $\pt M(g)$. Construct the surface $\Sigma'$ and the homeomorphism $\chi: \Sigma' \ra S_j$ as above. Consider the element of $\tilde {S_{\ast V}}$ associated to $g^\sharp$ by trivialization (\ref{trivtilde}). By applying the inverse of the chosen evaluation map ${\rm ev}$, we get a class in $H_0(S_j)/H_0^\sharp(S_j, V_j)$. Let $h \in H_0(S_j)$ be a representative in this class. The map $h^{-1} \circ \chi$ is again a homeomorphism from $\Sigma'$ to $S_j$. Let $f': V \ra \Sigma'$ be the marking map, which sends $v \in V$ to a point on $\Sigma'$ determined by $f$. By construction, $h^{-1} \circ \chi \circ f'$ sends every $v \in V$ to itself. Thereby, $h^{-1} \circ \chi$ associates the dual metric of $\Sigma$ to an element of $\ol{\mc D}^\sharp_M(S_j, V_j)$. By applying this procedure to each component of $\pt M(g)$, we obtain an element of $\ol{\mc D}^\sharp_M(\pt M, V)$, which we denote by $\mc I^\sharp_V(g^\sharp)$.

\begin{rmk}
\label{decktr}
The group of deck transformations of $\tilde{S_{\ast V}} \ra S_{\ast V}$ is naturally isomorphic to the pure braid group $B_0(\pt M, V)=H_0(\pt M, V)/H_0^\sharp(\pt M, V)$, see~\cite[Remark 3.5]{FP}. Thereby, it is also the group of deck transformations of the covering map $\ol{\mc P}^\sharp(M, V)\ra \ol{\mc P}(M, V)$. It is also the group of deck tranformations of the covering map $\ol{\mc D}^\sharp_M(\pt M, V) \ra \ol{\mc D}_M(\pt M, V)$. By construction, the map $\mc I^\sharp_V$ is equivariant with respect to $B_0(\pt M, V)$.
\end{rmk}

\begin{rmk}
Every geometric end $E_j$ of a convex cocompact manifolds $N(\ol g)$ is associated to a de Sitter spacetime $E^*_j$, which appears as the quotient of the subset of $\dS^3$ corresponding to the positive planes in a lift of $E_j$ to the universal cover. The function $\psi_j$ on $E_j$ determines the so-called \emph{cosmological time} on $E^*_j$. Instead of considering the surface $\Sigma'$, we may consider the surface $\Sigma^* \subset E^*_j$ dual to $\Sigma \subset E_j$. The induced metric on $\Sigma^*$ is exactly the dual metric of $\Sigma$. We, however, chose not to pursue this point of view, as some things become subtler from this perspective.
\end{rmk}


\subsection{Weak connectivity properties}

In this section we establish several weak connectivity properties of the spaces $\mc P^\sharp(M, V)$, $\mc D_M(\pt M, V)$ and $\mc D^\sharp_M(\pt M, V)$.

\begin{lm}
\label{connect}
Let $\hat g^\sharp_0$, $\hat g^\sharp_1 \in \mathcal P^\sharp(M, V)$. Then there exists $W \supseteq V$ and a path $g^\sharp_t \subset \overline{\mathcal P}^\sharp(M, W)$ for $t \in [0, 1]$ such that $g^\sharp_0$ and $g^\sharp_1$ are lifts of $\hat g^\sharp_0$ and $\hat g^\sharp_1$ respectively.
\end{lm}

\begin{proof}
The space $\mc{CH}^\sharp(N, V)$ is connected, so we connect $\hat g^\sharp_0$ and $\hat g^\sharp_1$ by a path $\hat g^\sharp_t$ in $\mc{CH}^\sharp(N, V).$ Denote by $(g_t, f_t) \in \mc{CH}(N, V)$ and by $\ol g_t \in \mc{CH}(N)$ the projections of $\hat g^\sharp_t$ to the respective spaces. Note that by definition of $\mc{CH}^\sharp(N, V)$, the marking maps $f_t$ are injective. The intersection of the immersed negative half-spaces of the planes $f_t(v)$, $v \in V$, is a totally convex set $K_t \subset N(\overline g_t)$. The boundary $\partial K_t$ never touches $C(\overline g_t)$ by construction. For $t=0,1$ the set $K_t$ coincides with $M(g_t)$. Moreover, near $t=0,1$ the set $K_t$ is homeomorphic to $M$ and $\hat g^\sharp_t \in \mathcal P^\sharp(M, V)$. We pick some $a, b \in (0,1)$, $a<b$, such that this holds for $t\in [0,a]$ and $t \in [b,1]$. What can go wrong with $K_t$ for $t \in [a,b]$? The issues are (1) a part of $K_t$ escapes outside $N(\ol g_t)$, so $K_t$ becomes non-compact; (2) a plane from $f_t(V)$ becomes not supporting to $K_t$.

Let us deal with the first issue. Let $E_t \subset N(\ol g_t)$ be the geometric end corresponding to the component $S_j$ of $\pt M$ and $\nu_j: \pi_1(S_j) \ra \pi_1(M)$ be the push-forward map. Develop the path $(g_t, f_t)$ to $\H^3$ and let $\rho_t$ be the holonomy map. Denote by $\Omega_t$ the component of $\pt_\infty \H^3 \backslash \Lambda(\rho_t)$ fixed by $\rho_t|_{\im(\nu_j)}$ and by $\tilde E_t$ denote the respective geometric end of $\tilde N(\ol g_t)$. Note that the map $t \mapsto \Lambda(\rho_t)$ is continuous with respect to the Hausdorff topology on the closed subsets of $\pt_\infty \H^3$, see~\cite[Theorem E]{AC}.
We choose a continuous family of compact fundamental domains $X_t$, $t\in [a,b]$, for $\Omega_t$, i.e., such that the map $t \mapsto X_t$ is continuous too. It is easy to deduce from the Hausdorff continuity that for every $t \in [a,b]$ there exist an interval $I_t \subset [a,b]$ and disjoint compact sets $Y_t$, $Z_t \subset \pt_\infty \H^3$ such that $I_t$ is open in $[a,b]$, contains $t$ and for every $t' \in I_t$ we have $X_{t'} \subset Y_t$ and $\Lambda(\rho_{t'}) \subset Z_t$. Pick a finite covering $I_{t_1}, \ldots, I_{t_r}$ of $[a, b]$ by such intervals $I_t$. For every $i=1, \ldots, r$ choose a covering of $Y_{t_i}$ by finitely many open disks in $\pt_\infty \H^3$ such that their closures are disjoint from $Z_{t_i}$. We now view all these disks as positive hyperbolic planes in $\tilde E_{t'}$, $t' \in I_{t_i}$. Consider their projection to $E_{t'}$. If necessary, we perturb them slightly so that no two planes have their closest points to $C(\ol g_{t'})$ on the same gradient line of the distance function to $C(\ol g_{t'})$. For the moment, every such plane exists only over the interval $I_{t_i}$, but we extend it arbitrarily continuously over all $[0,1]$ so that at $t=0,1$ the new planes are supporting to $K_t$. We do this for every $S_j$, and we mark these planes, together with the old ones, by a set $W \supseteq V$, extending the initial marking. Denote by $\hat K_t$ the modified sets $K_t$. By construction, every $\hat K_t$ is compact and homeomorphic to $M$.

We now need to resolve the second issue. To this purpose, for every $t$ we replace every plane $f_{t}(v)$ that is not supporting to $\hat K_t$ with the unique supporting plane on the same gradient line of the distance function to $C(\ol g_t)$. The obtained planes are moving continuously. This produces a path of metrics $g^\sharp_t \in \overline{\mathcal P}^\sharp(M, W)$ connecting $g^\sharp_0$ and $g^\sharp_1$ that are lifts of metrics $\hat g^\sharp_0$ and $\hat g^\sharp_1$.
\end{proof}


\begin{lm}
\label{connect1}
For every $V_0, V_1 \subset \pt M$ and $d_0 \in \mathcal D_M(\partial M, V_0)$, $d_1 \in \mathcal D_M(\partial M, V_1)$ there exist $W \supseteq (V_0 \cup V_1)$ and a path $\alpha: [0,1] \ra \ol{\mc D}_M(\pt M, W)$ such that $\alpha(0)$ and $\alpha(1)$ are lifts of $d_0$ and $d_1$ respectively, and that for all $t\in (0,1)$ we have $\alpha(t) \in \mc D_M(\pt M, W)$.
\end{lm}

\begin{lm}
\label{connect2}
For every continuous map $\beta: S^1 \rightarrow \overline{\mathcal D}^\sharp _M(\partial M, V)$ there exist $W \supseteq V$, a triangulation $\mc T$ of $(\pt M, W)$ and a continuous map $\alpha: \overline D^2 \rightarrow \ol{\mc D}^\sharp_M(\pt M, \mc T)$ of the closed 2-disk $\overline D^2$ with the boundary identified to $S^1$ such that $\alpha|_{S^1}$ is a lift of $\beta$, and that $\alpha(D^2) \subset \mathcal D^\sharp_M(\partial M, \mc T)$, where $D^2$ is the interior of $\overline D^2$.
\end{lm}

We note that Lemma~\ref{connect2} does not hold for spaces $\ol{\mc D}_M(\pt M, V)$. The reason is that not every loop in $\ol{\mc D}_M(\pt M, V)$ lifts to a loop in $\ol{\mc D}^\sharp_M(\pt M, V)$, and hence to a loop in $\ol{\mf D}_M(\pt M, V)$, while every loop in $\ol{\mc D}^\sharp_M(\pt M, V)$ lifts to a loop in $\ol{\mf D}_M(\pt M, V)$. This is actually the only reason why we introduce the spaces $\ol{\mc D}^\sharp_M(\pt M, V)$ and $\ol{\mc P}^\sharp(M, V)$. Lemma~\ref{connect1} obviously holds for spaces $\ol{\mc D}^\sharp_M(\pt M, V)$ too.

Denote by $\mathcal D^\sharp_{cl}(S, V)$ the space of isotopy classes of concave large spherical cone-metrics $d$ on $(S, V)$ with $V(d)=V$. Lemma~\ref{connect1} was first proven in~\cite[Section~9]{HR} for the spaces $\mathcal D^\sharp_{cl}(S, V)$ when $S$ is the 2-sphere.  Then Lemmas~\ref{connect1} and~\ref{connect2} for $\mathcal D^\sharp_{cl}(S, V)$ and arbitrary closed $S$ appeared in a manuscript~\cite{Sch5} of Schlenker. Brief proof sketches were given there. Basically, Schlenker noticed that a proof of Lemma~\ref{connect1} can be given exactly in the same way as it was done for spheres in~\cite{HR}. For Lemma~\ref{connect2} Schlenker remarks that a proof is an easy generalization of the proof of Lemma~\ref{connect1}. We claim that these results hold straightforwardly also for the spaces $\mathcal D^\sharp_M(\partial M, V)$. Indeed, we only need to modify the proofs to gain control of the smallest geodesic that is contractible not just in $\partial M$, but in $M$. For a metric $d$ on $S$ by ${\rm ml}(d)$ we denote the infimum of lengths of closed geodesics in $(S, d)$ that are contractible in $S$. The control on ${\rm ml}(d)$ in~\cite{HR} follows from two facts. The first is that if $\lambda>0$ is a number and $d$ is a Riemannian metric on $S$, then ${\rm ml}(\lambda d)=\lambda {\rm ml}(d)$, which clearly holds also for the function ${\rm ml}_M$. The second fact is that ${\rm ml}$ is upper-semicontinuous in the Lipschitz topology. For ${\rm ml}_M$ this was shown in Corollary~\ref{semicont}.

We now provide proof sketches of Lemmas~\ref{connect1} and~\ref{connect2}. All the details can be reconstructed with the help of~\cite{HR}.

\begin{proof}
We lift our initial classes of metrics ($d_0$ and $d_1$ for Lemma~\ref{connect1} or $\beta(S^1)$ for Lemma~\ref{connect2}) to $\mf D_M(\pt M, V)$ (where $V=V_0$ for $d_0$ and $V=V_1$ for $d_1$ in Lemma~\ref{connect1}). For the loop $\beta$ this is possible because $\mf D(\pt M, V) \cong \mc D^\sharp(\pt M, V) \times H^\sharp_0(\pt M, V)$, see Section~\ref{conemetsec}. We continuously smoothen these metrics in the class of Riemannian metrics with conical singularities and obtain in the end a pair or a loop of Riemannian metrics. As the set of Riemannian metrics on $S$ is connected and simply connected, we can connect our pair of smooth metrics by a path or we can contract the loop. Now we multiply each metric of our path or our disk by a scalar (equal 1 near the boundary of the path or of the disk) to achieve that the sectional curvature is bounded from above by 1 and that the length of the shortest geodesic loop in $\partial M$ that is contractible in $M$ is greater than some constant greater than $2\pi$. This is easy to perform due to the compactness of the path or the disk. For the involved cone-metrics we also take care that the cone-angles stay greater than $2\pi$ during our procedure. 

We approximate the obtained metrics by spherical cone-metrics. To this purpose one constructs a very small triangulation $\mc T$ of $\partial M$ with a vertex set $W$, which can be geodesically realized in all involved metrics. Next, for each metric we replace each triangle by a spherical triangle with the same lengths. If the triangulation was sufficiently fine, then the distances in the obtained cone-metrics got distorted by at most $\e$ comparing with the approximated metrics for an arbitrarily small globally defined $\e>0$. From comparison geometry one shows that the resulting cone-metrics are concave. Provided that $\e$ is sufficiently small, it also follows from the upper semi-continuity of the function ${\rm ml}_M$ that the length of the shortest geodesic in $\partial M$ that is contractible in $M$ remains greater than $2\pi$.
\end{proof}

It remains to make a few easy remarks on the topological boundary of $\mc D_c(\pt M, V)$ in $\mc D(\pt M, V)$. Let $d^\sharp \in \mc D^\sharp(\pt M, \mc T)$ and $\lambda \in \R$ be a number. We denote by $(\lambda *_{\mc T} d^\sharp)$ a metric obtained by changing the length $l_e$ of each edge $e \in E(\mc T)$ to $e^\lambda l_e$ if the lengths remain staying $<\pi$ under this operation (otherwise the operation is not defined). We will use some corollaries of the following easy lemma

\begin{lm}
Let $T$ be a spherical triangle with side lengths $a$, $b$ and $c$ less than $\pi$. Consider a spherical triangle $T'$ with the lengths $e^\lambda a$, $e^\lambda b$, $e^\lambda c$ for $\lambda > 0$ such that the side lengths of the new triangle are $<\pi$. Then the angles of $T'$ are bigger than the respective angles of $T$.
\end{lm}

The proof is identical to the proof of a similar lemma for hyperbolic triangles (\cite[Lemma 2.3.9]{Pro2}) with all inequalities reversed.

\begin{crl}
\label{multipl}
Let $d^\sharp \in \ol{\mc D}^\sharp_c(\pt M, \mc T)$ and $\lambda>0$ be sufficiently small. Then $(\lambda *_{\mc T} d^\sharp) \in \mc D^\sharp_c(\pt M, \mc T)$.
\end{crl}

This is straightforward since the only difference between $\ol{\mc D}^\sharp_c(\pt M, \mc T)$ and $\mc D^\sharp_c(\pt M, \mc T)$ is that in the former the cone-angles of some points of $V$ may be $2\pi$.

\begin{crl}
\label{locconnect}
Let $d \in \ol{\mc D} _M(\pt M, V)$. There exists an open neighborhood $U$ of $d$ in $\ol{\mc D} _M(\pt M, V)$ such that $U \cap {\mc D} _M(\pt M, V)$ is connected.
\end{crl}

\begin{proof}
Since the forgetful map $\ol{\mc D}^\sharp _M(\pt M, V) \ra \ol{\mc D} _M(\pt M, V)$ is a covering map, it is enough to prove Lemma~\ref{locconnect} for the space $\ol{\mc D}^\sharp _M(\pt M, V)$.
If $d^\sharp \in {\mc D}^\sharp _M(\pt M, V)$, then the claim is trivial. Suppose that $d^\sharp \notin {\mc D}^\sharp _M(\pt M, V)$. Note that $\ol{\mc D}^\sharp _c(\pt M, V)$ is locally determined by finitely many non-strict analytic inequalities in the manifold $\mc D^\sharp(\pt M, V)$. Hence, $\ol{\mc D}^\sharp _c(\pt M, V)$ is locally connected. Choose a triangulation $\mc T$ and a connected open neighborhood $U$ of $d^\sharp$ in $\ol{\mc D}^\sharp _c(\pt M, \mc T)$. Since ${\rm ml}_M(d^\sharp)>2\pi$ and ${\rm ml}_M$ is upper-semicontinuous, we can assume that $U \subset \ol{\mc D}^\sharp _M(\pt M, \mc T)$. Pick any two metrics $d_0, d_1 \in U \cap {\mc D}^\sharp _M(\pt M, V)$ and connect them by a path $d^\sharp_t$, $t \in [0,1]$, in $U$. Choose a function $\lambda: [0,1] \ra [0, \e]$ for some small $\e>0$ such that $\lambda(0)=\lambda(1)=0$ and $\lambda(t)>0$ for $t \in (0,1)$. Then, provided that $\e$ is sufficiently small, the path $(\lambda(t) *_{\mc T} d^\sharp_t)$ connects $d^\sharp_0$ and $d^\sharp_1$ in $U \cap {\mc D}^\sharp _M(\pt M, V)$.
\end{proof}



\subsection{End of the proof}

Now we can collect together the proof of the main result.

\begin{proof}[Proof of Theorem~\ref{mtr}]
Suppose that $\mc P(M, V)$ is non-empty. Consider the restriction of the map $\mc I_V$ to $\mathcal P(M, V)$. From Lemma~\ref{differ}, it is $C^1$. Due to Lemma~\ref{infrig}, the differential $d\mathcal I_V$ is non-degenerate. By the Inverse Function Theorem, $d\mathcal I_V$ is locally injective. Since the manifolds $\mathcal P(M, V)$ and $\mathcal D_M(\partial M, V)$ have the same dimensions, by Brouwer's Invariance of the Domain Theorem, the restriction of $\mathcal I_V$ is a local homeomorphism. Due to Lemma~\ref{proper}, it is proper. Thus, it is a covering map. We highlight that we can not really show the connectedness of $\mathcal P(M, V)$ and $\mathcal D_M(\partial M, V)$, so we take into account that it is a covering map between possibly disconnected spaces.

By Lemma~\ref{nonempty}, there exists some $V_0$ such that $\mc P(M, V_0)$ is non-empty. Let $d_0 \in \mc D_M(\pt M, V_0)$ be in the image of $\mc I_{V_0}$. Set $V_1:=V$ and pick an arbitrary $d_1 \in \mc D_M(\pt M, V_1)$. By Lemma~\ref{connect1}, there exists $W \supseteq (V_0 \cup V_1)$ and a path $\alpha: [0, 1] \ra \ol{\mc D} _M(\pt M, W)$ such that $\alpha(0)$ and $\alpha(1)$ are lifts of $d_0$ and $d_1$, and that the relative interior of the path belongs to $\mathcal D_M(\partial M, W)$. We have $\alpha(0)$ in the image of $\mc I_W$. We denote its preimage by $g_0$. We can modify the path so that for all small $t$ the metrics $\alpha(t)$ are in the image of $\mathcal I_W$. Indeed, consider any map $f_W: W \rightarrow F(g_0)$, where $F(g_0)$ is the set of faces of $g_0$, extending the marking map $f_0: V_0 \rightarrow F(g_0)$ of $g_0$. We can slightly perturb the respective faces by splitting them in parts to obtain a metric $g'_0 \in \mathcal P(M, W)$. By Corollary~\ref{locconnect}, there exists a neighborhood $U$ of $\alpha(0)$ such that $U \cap \mc D_M(\pt M, W)$ is connected. Hence, we can choose $g'_0$ so that $\mc I_W(g'_0)$ can be connected with the metric $\alpha(t)$ for some small $t$ by a path in $\mc D_M(\pt M, W)$. Taking the concatenation of this path with the rest of the initial path, we obtain a path satisfying our demand. Then the infinitesimal rigidity and the properness of $\mathcal I_W$ imply that all $\alpha(t)$ are in the image of $\mathcal I_W$. Hence, also $d_1$ is in the image of $\mc I_{V_1}=\mc I_V$, which means that $\mc P(M, V)$ is non-empty and $\mathcal I_V$ is surjective.

It remains to show that each metric in $\mathcal D_M(\partial M, V)$ has only one preimage. Consider a commutative diagram
\begin{center}
\begin{tikzcd}
\mc P^\sharp (M, V) \arrow[r, "\mc I^\sharp_V"] \arrow[d, "\tau^P_V"]
& \mc D^\sharp_M(\pt M, V) \arrow[d, "\tau^D_V"] \\
\mc P(M, V) \arrow[r, "\mc I_V"]
& \mc D_M(\pt M, V)
\end{tikzcd}
\end{center}
Here we denote by $\tau^P_V$ and $\tau^D_V$ the respective projection maps.
We already know that all the arrows except $\mc I^\sharp_V$ are covering maps. It follows that also $\mc I^\sharp_V$ is a covering map. We will show that it is a homeomorphism. The groups of the deck transformations of $\tau^P_V$ and $\tau^D_V$ are both naturally isomorphic to $B_0(\pt M, V)$ and $\mc I^\sharp_V$ is equivariant with respect to it, see Remark~\ref{decktr}. Thereby, this will imply that $\mc I_V$ is also a homeomorphism.

Suppose that for $g^\sharp_0, g^\sharp_1 \in \mathcal P^\sharp(M, V)$ we have $\mathcal I^\sharp_V(g^\sharp_0)=\mathcal I^\sharp_V(g^\sharp_1)$. Due to Lemma~\ref{connect}, there exists $W' \supseteq V$ and a path $\eta: [0,1] \ra \ol{\mc P}^\sharp(M, W')$ such that $\eta(0)$ and $\eta(1)$ are lifts of $g^\sharp_0$ and $g^\sharp_1$ respectively. Consider $S^1$ as $[0,1]$ with identified endpoints. Then $\mc I^\sharp_{W'}\circ \eta$ is a loop, which we denote by $\beta: S^1 \ra \ol{\mc D}^\sharp _M(\pt M, W')$. Due to Lemma~\ref{connect2}, there exist $W \supseteq W'$, a triangulation $\mc T$ of $(\pt M, W)$ and a map $\alpha: \ol D^2 \ra \ol{\mc D}^\sharp _M(\pt M, \mc T)$ such that $\alpha|_{S^1}$ is a lift of $\beta$, and that $\alpha(D^2) \subset \mathcal D^\sharp_M(\partial M, \mc T)$. 

There exists a lift $\zeta: [0,1] \ra \ol{\mc P}^\sharp(M, W)$ of the map $\eta$ such that $\alpha|_{S^1}=\mc I^\sharp_W \circ \zeta$. We claim that there exists a homotopy $\zeta_s$ of $\zeta_0=\zeta$ with $s \in [0, \e]$ for some small $\e>0$ such that for all $s>0$ we have $\zeta_s \subset \mc P^\sharp(M, W)$. Indeed, denote by $\ol g_t \in \mc{CH}(N)$ the projection of $\zeta(t)$. For $s>0$ and every $t \in [0,1]$ we move every immersed plane in $N(\ol g_t)$ marked by $v \in V$ at distance $s$ from its previous position along the segment connecting the closest point on this plane to the closest point on $C(\ol g_t)$ towards $C(\ol g_t)$. One can see that for sufficiently small $s$ the intersection of the immersed negative half-spaces of these planes defines a subset of $N(\ol g_t)$ homeomorphic to $M$, and that every marked plane intersects the boundary of this subset by a set with non-empty relative interior. This determines an element $\zeta_s(t) \in \mc P^\sharp(M, W)$.

Define $\theta_s(t):=\mc I^\sharp_W \circ \zeta_s(t)$. Since $\mc D^\sharp_M(\pt M, \mc T)$ is open, we may assume that $\e$ is sufficiently small so that $\theta_s(t) \in \mc D^\sharp_M(\pt M, \mc T)$ for all $t \in [0,1]$ and $s \in (0, \e]$.  Consider the arcs 
$$\theta'_s:=\{\theta_r(0), r\in [0, s]\} \cup \{\theta_r(1), r \in [0, s]\}.$$
For each $s$ the union $\theta_s \cup \theta'_s$ is a loop. We parametrize these loops as $\iota_s: S^1 \ra \ol{\mc D}^\sharp _M(\pt M, \mc T)$ so that the arc $\theta'_s$ corresponds to the restriction of $\iota_s$ to the interval $[1/2-s, 1/2+s]$ and $\iota_s(1/2)=\theta_0(0)=\theta_0(1)$.

Recall that for a number $\lambda$ and a $\mc T$-triangulable spherical cone-metric $d$ on $(\pt M, W)$ we denote by $(\lambda *_{\mc T} d)$ the metric obtained from $d$ by multiplication of all the side-lengths of $\mc T$ by $e^\lambda$. Consider a 1-parameter family of functions $\lambda_s: S^1 \ra [0, \e]$, where $s \in [0,\e]$, such that $\lambda_0$ is identically zero, and for all $s>0$ we have $\lambda_s(t)>0$ for $t \in (1/2-s, 1/2+s)$, but $\lambda_s$ is zero outside $(1/2-s, 1/2+s)$. Define loops $\delta_s:=(\lambda_s *_{\mc T} \iota_s)$. Provided that $\e$ is sufficiently small, for $s>0$ from Corollary~\ref{multipl} we get that $\delta_s \subset \mc D^\sharp_M(\pt M, \mc T)$. We note that the loop $\delta_0$ is a reparametrization of $\alpha|_{S^1}$.

Because $\delta_0$ is contractible in $\ol{\mc D}^\sharp _M(\pt M, \mc T)$, so are the loops $\delta_s$. We note that this means that for $s>0$ the loops $\delta_s$ are contractible in ${\mc D^\sharp} _M(\pt M, \mc T)$. Indeed, let $\alpha_s: \ol D^2 \ra \ol{\mc D}^\sharp _M(\pt M, \mc T)$ be a homotopy contracting $\delta_s$ in $\ol{\mc D}^\sharp _M(\pt M, \mc T)$. Choose a function $\lambda_s: \ol D^2 \ra [0, \e]$ such that $\lambda_s|_{S^1}=0$ and $\lambda_s|_{D^2}>0$. Consider now a homotopy $\alpha'_s:=(\lambda_s *_{\mc T} \alpha_s)$. By Corollary~\ref{multipl}, provided that $\e$ is sufficiently small, we get
$$\alpha'_s(D^2) \subset {\mc D^\sharp} _M(\pt M, \mc T).$$
Because the restriction of $\mc I^\sharp_W$ to $\mc P^\sharp(M, W)$ is a covering map and the loops $\delta_s$ for $s>0$ are contractible in $\mc D^\sharp_M(\pt M, W)$, they lift to loops in $\mc P^\sharp(M, W)$. This means that for $s>0$ there are paths $\kappa_s: [0, 1] \ra \mc P^\sharp(M, W)$ with $\kappa_s(0)=\zeta_s(0)$, $\kappa_s(1)=\zeta_s(1)$ such that 
$$\mc I^\sharp_W \circ \kappa_s(t) = \delta_s(1/2-s+2ts).$$
In particular, the arcs $\mc I^\sharp_W \circ \kappa_s$ converge to $\delta_0(1/2)$ as $s$ tends to zero.

We now project the situation to the spaces $\ol{\mc P}(M, W)$ and $\ol{\mc D}_M(M, W)$. Define $\xi_s:=\tau^D_W\circ\delta_s$ and $\omega_s:=\tau^P_W\circ\kappa_s$ (here we consider $\ol{\mc P}^\sharp(M, W)$ and $\ol{\mc D}^\sharp_M(M, W)$ as the domains of $\tau^P_W$ and $\tau^D_W$). Consider a compact neighborhood $U$ of $\xi_0(1/2)$ such that all arcs obtained from the restrictions of $\xi_s$ to $[1/2-s, 1/2+s]$ belong to $U$. By Lemma~\ref{proper}, $\mc I_W$ is proper, thus, the set $X:=(\mc I_W)^{-1}(U)$ is compact. (We note that at this point we do not know whether $\mc I^\sharp_W$ is proper, this is the reason why we had to use $\mc I_W$.) Hence, there exists a Hausdorff limit $\omega \subset X$ of the arcs $\omega_s$ (here we perceive $\omega_s$ and $\omega$ not as maps, but as point-sets). By continuity, $\mc I_W(\omega)=\xi_0(1/2)$. Because $\mc I_V$ is a local homeomorphism, the points $\tau^P_W\circ\zeta(0)$, $\tau^P_W\circ\zeta(1)$ are isolated in $(\mc I_W)^{-1}(\xi_0(1/2))$. Hence, $\omega$ contains isolated points. This can not happen as $\omega$ is the Hausdorff limit of the connected arcs $\omega_s$ and $\ol{\mc P}(M, W)$ is a Hausdorff space. This contradiction finishes the proof.

\end{proof}

\bibliographystyle{abbrv}
\bibliography{dual}
\bigskip{\footnotesize\par
  \textsc{University of Vienna, Faculty of Mathematics, Oskar-Morgenstern-Platz 1, A-1090 Vienna, Austria} \par
  \textit{E-mail}: \texttt{roman.prosanov@univie.ac.at}
}

\end{document}